\documentclass[reqno,11pt]{amsart}
\usepackage{a4wide,color,eucal,enumerate,mathrsfs}
\usepackage[normalem]{ulem}
\usepackage{amsmath,amssymb,epsfig,amsthm} 
\usepackage[latin1]{inputenc}
\usepackage{psfrag}

\def\U{{\mathcal U}}

\def\az{\alpha}

\def\dist{{\mathop\mathrm{\,dist\,}}}

\def\ez{\epsilon}

\newcommand{\average}{{\mathchoice {\kern1ex\vcenter{\hrule height.4pt
width 6pt depth0pt} \kern-9.7pt} {\kern1ex\vcenter{\hrule
height.4pt width 4.3pt depth0pt} \kern-7pt} {} {} }}
\newcommand{\ave}{\average\int}

\def\bint{{\ifinner\rlap{\bf\kern.35em--}
\int\else\rlap{\bf\kern.45em--}\int\fi}\ignorespaces}

\def\bbint{{\ifinner\rlap{\bf\kern.35em--}
\hspace{0.078cm}\int\else\rlap{\bf\kern.45em--}\int\fi}\ignorespaces}

\def\diam{{\mathop\mathrm{\,diam\,}}}

\newcommand{\R}{\mathbb{R}}

\newtheorem{thm}{Theorem}[section]
\newtheorem{lem}[thm]{Lemma}
\newtheorem{prop}[thm]{Proposition}
\newtheorem{defn}[thm]{Definition}
\numberwithin{equation}{section}

\theoremstyle{remark}
\newtheorem{rem}[thm]{Remark}

\def\bint{{\ifinner\rlap{\bf\kern.35em--}
\int\else\rlap{\bf\kern.45em--}\int\fi}\ignorespaces}

\usepackage{graphicx}
\usepackage{import}
\usepackage{xifthen}
\usepackage{pdfpages}
\usepackage{transparent}

\title[Finite Morse index solutions to supercritical elliptic PDEs]{Uniform boundedness for
finite Morse index solutions\\ to supercritical semilinear elliptic  equations}
\author{Alessio Figalli, Yi Ru-Ya Zhang}
\date{\today}

\address{ETH Z\"urich, Department of Mathematics, R\"amistrasse 101, 8092, Z\"urich, Switzerland}
\email{alessio.figalli@math.ethz.ch}  

\address{Academy of Mathematics and Systems Science, the Chinese Academy of Sciences, Beijing 100190, China}
\email{yzhang@amss.ac.cn}

 \thanks{Both authors have received funding from the European Research Council under the Grant Agreement No. 
721675 ``Regularity and Stability in Partial Differential Equations (RSPDE)''. The second author is also partially funded by the Chinese Academy of Science and NSFC grant No. 11688101}

\begin{document}

\begin{abstract}
We consider finite Morse index solutions to semilinear elliptic questions, and we investigate their smoothness.
It is well-known that:\\
- For $n=2$, there exist Morse index $1$ solutions whose $L^\infty$ norm goes to infinity.\\
- For $n \geq 3$, uniform boundedness holds in the subcritical case for power-type nonlinearities, while for critical nonlinearities the boundedness of the Morse index does not prevent blow-up in $L^\infty$.\\
In this paper, we investigate the case of general supercritical nonlinearities inside convex domains, and we prove an interior a priori $L^\infty$ bound for finite Morse index solution in the sharp dimensional range $3\leq n\leq 9$.
As a corollary, we obtain uniform bounds for finite Morse index solutions to the Gelfand problem constructed via the continuity method.
\end{abstract}


\maketitle

\section{Introduction}
Given $\Omega \subset \R^n$ a bounded domain, and $f:\R\to \R$ a nonnegative $C^1$ function, we consider 
a solution  $u:\Omega \to \R$ to the following semilinear equation
\begin{equation}\label{equ}
\left\{
\begin{array}{ll}
-\Delta u= f(u) &\text{ in } \Omega,\\
 u=0 &\text{ on } \partial\Omega.
 \end{array}
 \right.
\end{equation}
Note that, by the nonnegativity of $f$ and the maximum principle, $u> 0$ inside $\Omega$ (unless $u \equiv 0$).

Set $F(t):=\int_0^tf(s)\,ds$. Then \eqref{equ} corresponds to the Euler-Lagrange equation for the energy functional
$$
\mathcal E[u]:=\int_{\Omega}\Bigl(\frac{|\nabla u|^2}2-F(u)\Bigr)\,dx.
$$
Consider the second variation of $\mathcal E$, that is,
$$
\frac{d^2}{d\epsilon^2}\Big|_{\epsilon=0}\mathcal E[u+\epsilon\xi]
=\int_{\Omega}\Bigl(|\nabla \xi |^2-f'(u)\xi^2\Bigr)\,dx.
$$
Given a subdomain $\Omega'\subseteq \Omega$ and $k \in \mathbb N$,
$u$ is said to have {\em finite Morse index} $k$ in $\Omega'$, and we write \ ${\rm ind}(u,\Omega')=k$, if $k$ is the maximal dimension of a subspace $X_k\subset C^1_c(\Omega')$ such that, for any $\xi\in X_k\setminus \{0\}$,
$$
Q_u(\xi) :=\int_{\Omega'}\Big(|\nabla\xi|^2 -f'(u)\xi^2\Big)\, dx<0.
$$
Also, $u$ is said to be {\em stable} in $\Omega'$ if ${\rm ind}(u,\Omega')=0$ (that is, $Q_u(\xi)\geq 0$ for all $\xi \in C^1_c(\Omega')$).

\subsection{Finite Morse index vs uniform boundedness}
The idea of using a bound on the Morse index to characterize the uniform boundedness of a solution to a semilinear elliptic equation was first introduced in the seminal paper \cite{BL1992}.
In this work, as well in several other subsequent papers (see for instance \cite{Y1998,HRS1998}),
the authors considered  (variants of) the subcritical case, namely 
$$f(t) \simeq (\alpha+ t)^{\frac {n+2}{n-2}-\ez},\qquad \alpha, \ez>0,$$
and they proved that the boundedness of solutions is equivalent to the boundedness  of the Morse index.\footnote{Note however that this result does not cover the full subcritical case: as shown in \cite[Lemma 5 and Theorem 1(iii))]{JL1972}, for $f_\lambda(t)=\lambda (1+t)^p$, $\lambda>0$ and $p < \frac{n+2}{n-2}$, there exists a family $u_\lambda$ of solutions with Morse index $1$ with $\|u_\lambda\|_{L^\infty(B_1)}\to \infty$ as $\lambda \to 0$. In other words, in the subcritical case, both upper and lower bounds are needed on $f$ in order to show the equivalence between boundedness in $L^\infty$ and boundedness of the Morse index.}

In the critical case, namely
$$f(u)=(\alpha+u)^{\frac{n+2}{n-2}}, \qquad \alpha>0,$$
the finiteness of the Morse index does not imply the boundedness of the solutions. Indeed it is not difficult to check that the functions
\begin{equation}\label{critical case}
u_{\alpha,\mu}(x)=\biggl(\bigg(\frac {\mu \sqrt{n(n-2)}}{\mu^2+|x|^2}\bigg)^{\frac{n-2}{2}}-\alpha\biggr)_+, \ \mu>0
\end{equation}
are solutions with Morse index $1$. In particular, choosing $\alpha_\mu:=\left(\frac {\mu \sqrt{n(n-2)}}{1+\mu^2}\right)^{\frac{n-2}{2}}$ so that $u_{\alpha_\mu,\mu}=0$ on $\partial B_1$, and letting $\mu\to 0$, one can construct a family Morse index $1$ solutions in $B_1$ whose $L^\infty$ norm goes to infinity (see, e.g., \cite{CGS1989}).

The supercritical case, instead, is much less understood. The special case where $f$ is a polynomial or an exponential function has been studied in \cite{F2007} and \cite{DF2009}, respectively. There, the uniform boundedness of solutions is obtained by proving suitable Liouville-type results. 
Unfortunately, this approach does not seem suitable for general nonlinearities.

We finally mention a recent result \cite{DF2021}, where the authors investigate the regularity and symmetry properties of finite Morse index solutions.

\subsection{Main result: finite Morse index solutions with supercritical nonlinearities are uniformly bounded}
Very recently, in  \cite{CFRS2020} the authors investigated the properties of stable solutions
for {\it all} nonlinearities, and they proved uniform boundedness when $3 \leq n \leq 9$,\footnote{
In the stable case, the case $n=2$ is a consequence of the case $n=3$ by noticing that a stable solution in two dimensions is also stable in three dimensions (by looking at it as a function constant in the third variable). This is why the results in \cite{CFRS2020} hold for $n\leq 9$. This is not the case anymore when ${\rm ind}(u)>0$, and indeed a change of behavior of finite Morse index solutions between dimension $n=2$ and dimension $3 \leq n \leq 9$ was already observed in \cite{JL1972}. In particular, as \cite[Fig. 1 b, pag. 245]{JL1972} shows, there exists a curve of two-dimensional solutions with Morse index $1$ that blows-up in $L^\infty$.} and interior $W^{1,2}$ estimates in all dimensions.

In this paper, we exploit these result to develop a series of new tools for finite Morse index solutions (cf. Section~\ref{sect:pf main thm} below) that allow us to prove a universal $L^\infty$ bound for solutions to  \eqref{equ} when $f$ grows superlinearly in a suitably quantified way.\footnote{Our quantitative superlinarity assumption \eqref{F and f} already appeared in the paper \cite{S2000} (see also \cite{DS2010}), where the author proved the uniqueness of solutions to \eqref{eq:f l} for small values of $\lambda$.}  As common in these problems, we assume that $f(0)>0$ (actually, we quantify this assumption by asking that $f(0)\geq c_0>0$, so to better emphasize the dependences in our $L^\infty$ bound). This assumption is particularly natural in the superlinear case, since the Derrick-Pohozaev identity prevents the existence of nontrivial solutions (see \cite[Theorem 1, Page 515]{E2010}).

Actually, because of applications to the Gelfand problem described in Section~\ref{sect:Gelfand} below, it will be convenient to prove a more robust result that establishes a uniform bound whenever the nonlinearity is of the form $\lambda f$, where $\lambda \in [0,\hat{\lambda}]$ for some fixed $\hat\lambda$.
Also, for the sake of generality, it is interesting to observe how the bound depends on $f$. So, instead of considering a fixed nonlinearity $f$, we assume that $f$ belongs to a locally compact $C^1$ family. As shown in Section ~\ref{sect:convex} below,
this assumption can be considerably weakened if $f$ is assumed to be convex.

\begin{thm}\label{mainthm 1}
Let $3\le n\le 9$,  $\Omega\subset \R^n$ a bounded convex domain, and $c_0>0$. Consider
$$
\mathcal K\subset \{h \in C^1(\R):\,h \geq 0,\,h' \geq 0, h(0)\geq c_0\},
$$
and assume that $\mathcal K$ is compact for the $C^1_{\rm loc}(\R)$ topology.
Let $\hat{\lambda} >0$, and let $u \in C^2(\Omega)$ solve 
\begin{equation}
\label{eq:f l}
\left\{
\begin{array}{ll}
-\Delta u= \lambda f(u) &\text{ in } \Omega,\\
 u=0 &\text{ on } \partial\Omega,
 \end{array}
 \right.
\end{equation}
 for some $f \in \mathcal K$ and $\lambda \in [0,\hat{\lambda} ]$.
 Finally, assume that ${\rm ind}(u,\Omega)\leq k$ and that there exist  $\ez,t_0>0$ such that 
\begin{equation}\label{F and f}
f(t)t \geq \left (\frac{2n}{n-2}+ \ez  \right)F(t) \quad \text{ for all }  \  t \geq t_0,
\end{equation}
where $F(t):=\int_0^t f(s)\,ds$.
Then
$$\|u\|_{L^\infty(\Omega)}\leq C(n,k,\mathcal K,\hat{\lambda} ,\ez,t_0,\Omega).$$
\end{thm}

\begin{rem}
\label{rmk:superlin}
We observe that, as a consequence of \eqref{F and f}, it holds
\begin{equation}\label{supercritical}
  f(t)\geq c_1 t^{\frac{n+2}{n-2}+\ez}\qquad \forall\,t \geq 0,\end{equation}
  with $c_1:=f(0) t_0^{-\frac{n+2}{n-2}-\ez}$.

Indeed, \eqref{F and f} can be rewritten as
$$
F'(t) \geq \frac{\left (\frac{2n}{n-2}+ \ez  \right)}{t}F(t) \quad \text{ for all }  \  t \geq t_0,
$$
so it follows from Gr\"onwall inequality that
$$
F(t)\geq F(t_0)\left(\frac t{t_0}\right)^{\frac{2n}{n-2}+\ez}.
$$
Inserting this information in \eqref{F and f}, we get
$$f(t)\geq \left(\frac{2n}{n-2}+\ez \right)F(t_0)t_0^{-\frac{2n}{n-2}-\ez} t^{\frac{n+2}{n-2}+\ez}    \quad \text{ for all }  \   t \geq t_0.$$
Also, since $f$ is increasing we have $F(t_0)\geq f(0)t_0,$ and therefore
$$f(t)\geq 
\left\{
\begin{array}{cc}
f(0) \left(\frac{2n}{n-2}+\ez \right) t_0^{-\frac{n+2}{n-2}-\ez} t^{\frac{n+2}{n-2}+\ez}    & \text{ for}  \   t \ge t_0\\
f(0) & \text{ for }  \   0\le t < t_0,
\end{array}
\right. 
$$
which implies \eqref{supercritical}.
\end{rem}

\begin{rem}
As mentioned before, the dimensional range $3\le n \le 9$ follows from \cite[Theorem 1.2]{CFRS2020}, since boundedness of stable solutions for all nonlinearities is true only under this assumption.
However, for some particular choices of nonlinearities (e.g., $f(u)=(1+u)^p$ for suitable values of~$p$), we believe that our ideas and techniques could be applied also in higher dimension (cf. \cite{DDF2011}).
\end{rem}

\subsection{About Theorem~\ref{mainthm 1}: extensions and tools used in the proof}
We first discuss some possible extensions and generalizations of Theorem~\ref{mainthm 1},
and then we briefly present the three key ingredients behind its proof.

\subsubsection{On the convexity of $\Omega$}
The convexity assumption on $\Omega$ in Theorem~\ref{mainthm 1} allows us:\\
- to focus only on interior regularity, since the regularity near the boundary is handled via the moving plane method, see Lemma~\ref{stable outside K} below;\\
- to apply the classical Derrick-Pohozaev argument on convex domains, see the argument after \eqref{vs}. 

We believe that a nontrivial modification of our techniques could be used to analyze the boundary behavior inside general smooth domains.
However our proof strongly relies on the Derrick-Pohozaev argument, and this requires $\Omega$ to be at least star-shaped (see, e.g., \cite[Theorem 1, Page 515]{E2010}).
Hence, it looks likely to us that by combining the ideas developed in this paper with the boundary regularity from \cite{CFRS2020}, one should be able to extend Theorem~\ref{mainthm 1} to (sufficiently smooth) star-shaped domains.

\subsubsection{A result for convex nonlinearities}
\label{sect:convex}
The assumption that the nonlinearity $f$ belongs to a family $\mathcal K$ that is compact for the $C^1_{\rm loc}(\R)$ topology can be removed, if one assumes the nonlinearities to be convex and to be dominated by a fixed  continuous nonnegative function $g:\R\to \R$.
More precisely, if $f$ is a convex function such that $0\leq f \leq g$ and \eqref{F and f} holds, then $\|u\|_{L^\infty(\Omega)}\leq C(n,k,g,\hat\lambda,\ez,t_0,\Omega).$

Indeed, the compactness assumption in $C^1_{\rm loc}$ is used only to apply Proposition~\ref{prop:stability}.
If all the nonlinearities are convex, then the bound $0\leq f \leq g$ guarantees compactness in $C^0_{\rm loc}$. Therefore, one only needs to check that Proposition~\ref{prop:stability} holds if $f_j$ are convex functions satisfying $f_j \to f_\infty$ in $C^0_{\rm loc}(\R)$. This can be done by suitable adapting the notion of stability for convex functions, defining $$Q_u(\xi) :=\int_{\Omega'}\Big(|\nabla\xi|^2 -f_-'(u)\xi^2\Big)\, dx\qquad \text{with $f'_-(t):=\lim_{\tau\to 0^+}\frac{f(t)-f(t-\tau)}{\tau}=\sup_{\tau>0}\frac{f(t)-f(t-\tau)}{\tau}$}.
$$
Indeed, with this definition, the results from \cite{CFRS2020} still apply. In addition, the following lower semicontinuity property holds:
$$
t_j \to t,\quad f_j \to f \text{ in }C^0_{\rm loc}(\R)\qquad \Rightarrow\qquad f'_-(t) \leq \liminf_{j\to \infty}(f_j)_-'(t_j),
$$
and this allows one to show that upper bounds on the Morse index are preserved. We leave the details to the interested reader.

\subsubsection{Main tools}
\label{sect:pf main thm}
As mentioned before, the proof of Theorem~\ref{mainthm 1} is based on a series of new important results for finite Morse index solutions. These are:\\
(i) A general stability result for bounded Morse index solutions stating that, for $3 \leq n \leq 9$, these families are weakly compact in $W^{1,2}$ and they converge in $C^2_{\rm loc}$ outside finitely many points (see Proposition~\ref{prop:stability}). This result relies on the smoothness of stable solutions for $n\leq9$ obtained in \cite{CFRS2020}, and on a slight improvement of it proved in Appendix A.\\
(ii) A uniform $W^{1,2}$ integrability estimate for  finite Morse index solutions (see Proposition~\ref{uniform integrable}). This result depends both on the supercriticality assumption \eqref{F and f} and on the interior $W^{1,2}$ estimates for stable solutions, cf. \cite{CFRS2020}.\\
(iii) A $\varepsilon$-regularity theorem for finite Morse index solutions stating that, if the $W^{1,2}$ norm of a solution inside a ball $B_r$ decays sufficiently fast for $r \in [\varepsilon,1]$ with $\varepsilon \ll 1$, then it decays all the way to the origin (see Proposition~\ref{prop:eps reg}).\\
It is worth observing that while (i) needs the dimensional restriction $n \leq 9$, both (ii) and (iii) hold in every dimension.
Besides playing a crucial role in proving Theorem~\ref{mainthm 1}, we believe that these results have their own interest.

\subsection{An application to the Gelfand problem associated to analytic supercritical nonlinearities}
\label{sect:Gelfand}
Given $f:\R\to \R$ nonnegative and increasing, and $\lambda\geq 0,$ the  so-called Gelfand problem for 
 $f$ consists in studying the nonlinear elliptic problem
$$
\left\{
\begin{array}{ll}
-\Delta u= \lambda f(u) &\text{ in } \Omega,\\
 u=0 &\text{ on } \partial\Omega.
 \end{array}
 \right.
$$
This problem has a long history: it was first presented by Barenblatt in a volume edited by Gelfand \cite{G1963}, and a series of authors studied it later, in particular in the range where $u$ is stable; we refer the interested reader to \cite{BE1989,B2003,BV1997,D2011,C2017} for a complete account on this topic.

In this paper we want to study the solution curve associated to the Gelfand problem: we look for a continuous curve $\lambda\colon [0,\infty)\to [0,\infty)$ with $\lambda(0)=0,$ and for a one-parameter family of solutions $\{u_s\}_{s\geq 0}$, such that
$$
\left\{
\begin{array}{ll}
-\Delta u_s= \lambda(s)f(u_s) &\text{ in } \Omega,\\
 u_s=0 &\text{ on } \partial\Omega.
 \end{array}
 \right.
$$
When $\Omega=B_1$, the cases $f(t)=(1+\az t)^\beta,\,\az,\beta>0,$ and $f(t)=  e^t$, have been fully understood in \cite{JL1972} via ODE methods. In particular, when $3\le n\le 9$, the authors proved that there are infinitely many turning points in the solution curve $s\mapsto (\lambda(s),\|u_s\|_{L^\infty(\Omega)},)$ for suitable values of $\beta$. Later, similar phenomena were observed for special functions $f$ or in low dimensional domains with suitable symmetries (see, e.g., \cite{NS1994,D2000,GW2008,DF2009, K2014} and the reference therein). 

Assume now that $\Omega$ is a convex set of class $C^3$, let $C^1_0(\overline\Omega)$ denote the Banach space of $C^1$ functions on $\overline\Omega$ that vanish on $\partial\Omega$,
and consider the following open subset of $C^1_0(\overline\Omega)$ endowed with the $C^1$ topology:
\begin{equation}
\label{eq:O}
\mathcal O:=\{u\in C^1_0(\overline\Omega)\,:\,u>0 \text{ in }\Omega,\,\partial_\nu u|_{\partial\Omega}<0\},
\end{equation}
where $\nu$ denotes the outer unit normal to $\partial\Omega$.
Following \cite{D2000}, assume the map
$$
\mathcal O\ni u \mapsto f(u) \in C^0(\overline\Omega)
$$
to be real analytic (as noted in \cite{D2000}, this is the case for instance if $f$ is analytic).
Then, thanks to our Theorem~\ref{mainthm 1}, one can apply the global analytic bifurcation theory developed in \cite[Section 2.1]{BDT2000}
to show the existence of a piecewise analytic continuous curve $[0,\infty)\ni s \mapsto (\lambda(s),u_s)$, with $(\lambda(0),u_0)=(0,0)$, such that both $\|u_s\|_{L^\infty(\Omega)}$ and ${\rm ind}(u_s,\Omega)$ tend to infinity as $s \to \infty$.
Moreover there exists a sequence $(\lambda(s_i),u_{s_i})$ such that $\|u_{s_i}\|_{L^\infty(\Omega)} \to \infty$
and each point of this sequence is either a bifurcation or a turning point.
This is a complete statement:

\begin{thm}\label{mainthm 2}
Let $3\le n\le 9$, $\Omega\subset \mathbb R^n$ a bounded convex domain of class $C^3$, and $f> 0$ an increasing analytic function satisfying \eqref{F and  f}. Let
$$\mathcal S:=\{(\lambda,u)\in  \mathbb R_+ \times C_0^1(\overline\Omega) \colon -\Delta u - \lambda f(u)=0\,\text{ and }\, -\Delta - \lambda f'(u) \text{ is invertible with bounded inverse}\}. $$
Then there exist two continuous mappings 
$$h_1\colon \mathbb R_+ \to \mathbb R_+,\qquad h_2\colon \mathbb R_+ \to C_0^1(\overline\Omega), $$
 so that, denoting $h=(h_1,\,h_2)$, we have:
\begin{enumerate}
\item[(i)] $h \colon \mathbb R_+ \to \overline{\mathcal{S}}$ and $\lim_{s\to \infty} {\rm ind}(h_2(s),\Omega)=\infty.$
\item[(ii)]  $h$ is injective on $h^{-1}(\mathcal{S})$ with $h_1'(s)\neq 0$, and real analytic at all points $s\in h^{-1}(\mathcal{S})$.
\item[(iii)]  The set  $h^{-1}(\overline{\mathcal{S}}\setminus \mathcal{S})$ consists of isolated values.
\item[(iv)] For every point $s_0\in h^{-1}(\overline{\mathcal{S}}\setminus \mathcal{S})$ there exists an injective and continuous reparameterization $s=\gamma(\sigma)$, $\sigma \in [-1,1]$, such that $s_0=\gamma(0)$
and $h\circ \gamma$ is a real analytic function whose derivatives might only vanish at $0$.
\item[(v)] There are infinitely many values of $s>0$ where $h(s)\in \overline{\mathcal{S}}\setminus \mathcal{S}$ is either a bifurcation or a turning point. Namely, either in every neighborhood of $h(s)$ there exists a solution of \eqref{equ} which is not in the image of $h$ (and then $h(s)$ is a bifurcation point), or the previous case do not happen but $h_1$ is not locally injective (and then $h(s)$ is a turning point).
\end{enumerate}
\end{thm}

\begin{proof}
Let $\mathcal O$ be as in \eqref{eq:O}, and define the analytic map
$$\mathcal F\colon \mathbb R_+\times \mathcal O \to C^1_0(\overline\Omega), \qquad  \mathcal F(\lambda,\, u) :=  - u + \lambda \mathcal A(u),$$
where $\mathcal A(u):=(-\Delta)^{-1}[f(u)]$, and $(-\Delta)^{-1}$ denotes the inverse of the Dirichlet Laplacian in $\Omega$.
Arguing exactly as in the proof of \cite[Theorem 1]{D2000} (see also the remark after the statement of the theorem), the result follows from \cite[Section 2.1]{BDT2000}.
\end{proof}

\begin{rem}
It was pointed out in \cite[Remark 4]{D2000} that, when $\Omega$ is a $C^3$ strongly convex domain with certain symmetries, a careful modification of \cite{ST1979}  gives that, the image of $h$ is a smooth curve with only infinitely many turning points but not bifurcation points. Also, this property in generic in a neighborhood of such domains.
We expect a similar result to hold also in our setting. 
\end{rem}

\subsection{Structure of the paper}
The paper is organized as follows. In Section 2 we present a series of results on finite Morse index solutions, which will be crucial for proving Theorem~\ref{mainthm 1}. Then, in Section 3 we prove Theorem~\ref{mainthm 1}. Finally, in a first appendix, we show that \cite[Theorem 1.2]{CFRS2020} holds also for $W^{1,2}$ stable solution that are $C^2$ outside one point. This result is used in the proof of Proposition~\ref{prop:stability}.
Then, in a second appendix, we describe how the method in \cite{CFRS2020} implies uniform boundedness of solutions whenever the spectrum of $-\Delta -f'(u)$ is bounded from below.

\medskip
\noindent
{\it Acknowledgments.} The authors are grateful to Xavier Cabr\'e and Alberto Farina for useful comments on a preliminary version of this manuscript.

\section{Technical tools on finite Morse index solutions}
Let us fix some notation. 
For $x\in \mathbb R^n$ and $r>0$, we denote by $B_r(x)$ the Euclidean ball centered at $x$ with radius $r$. The center is usually omitted when $x$ is the origin. By $\az B$ we mean the ball  with the same center as $B$ but $\az$ times its radius. We  write  constants as positive real numbers $C(\cdot)$, with the parentheses including all the parameters on which the constants depend. We note that $C(\cdot)$ may vary between appearances, even within a chain of inequalities. Sometimes we use $C_n, c_n$ to emphasize that a constant depends only on the dimension.

The goal of this section is to prove several new important results on finite Morse index solutions that will be used in the next section to prove Theorem~\ref{mainthm 1}.
First, we need to introduce a notion of weak solution with bounded Morse index. 
\begin{defn}
Let $\U\subset \R^n$ be an open set, and let $f:\R\to \R$ be nonnegative.
We say that  $u \in W^{1,2}_{\rm loc}(\U)$ is a weak solution of $-\Delta u=f(u)$ in $\U$ if $f(u) \in L^1_{\rm loc}(\U)$ and
$$
\int_{\U}\nabla u\cdot \nabla \varphi \,dx=\int_{\U}f(u)\,\varphi\,dx \qquad \forall\,\varphi \in C^1_c(\U).
$$
Assume in addition that $f$ is of class $C^1$. Then we say that $u$
has {\em finite Morse index} $k \in \mathbb N$ in $\U$, and we write ${\rm ind}(u,\U)=k$, if $f'(u) \in L^1_{\rm loc}(\U)$ and $k$ is the maximal dimension of a subspace $X_k\subset C^1_c(\U)$ such that
$$
Q_u(\xi) :=\int_{\U}\Big(|\nabla\xi|^2 -f'(u)\xi^2\Big)\, dx<0 \qquad \forall\,\xi \in X_k \setminus \{0\}.
$$
\end{defn}

As we shall see below, whenever $f' \geq 0$ it is possible to prove an a priori bound on the $L^1_{\rm loc}$ norm of $f'(u)$ in terms of the Morse index.
Then, by Fatou's Lemma, this a priori bound holds for all weak solutions that are limits of smooth solutions (see Proposition~\ref{prop:stability} below).

\begin{lem}
\label{lem:Morse}
Let $\U\subset \R^n$ be an open set, let $f:\R\to \R$ be nonnegative, increasing, and of class $C^1$, and let $u \in W^{1,2}_{\rm loc}(\U)$ be a weak solution of $-\Delta u=f(u)$ in $\U$ with ${\rm ind}(u,\U)\leq k$.
Then:
\begin{itemize}
\item[(i)] 
If $\{\U_i\}_{i=1}^{k+1}$ is a disjoint family of open subsets of $\U$, then $u$ is stable in at least one set $\U_i$.
\item[(ii)] The following uniform bound holds:
\begin{equation}
\label{eq:L1 f'u}
\int_{B_r(\bar x)}f'(u)\,dx \leq C_n(1+k)^{\frac2n}r^{n-2} \qquad \forall\,B_{2r(\bar x)}\subset \U.
\end{equation}
\end{itemize}

\end{lem}
\begin{proof}
To prove (i) we note that, if by contradiction $u$ was unstable inside each set $\U_i$, then there would exist functions $\xi_{i} \in C^1_c(\U_i)$ such that
$$
\int_{\U}\left(|\nabla \xi_i|^2-f'(u)\xi_i^2\right)\,dx<0.
$$
Since the functions $\{\xi_{i}\}_{i=1}^{k+1}$ have disjoint support, this implies that
$$
\int_{\U}\left(|\nabla \xi|^2-f'(u)\xi^2\right)\,dx<0 \qquad \forall\, \xi \in {\rm Span}(\xi_{1},\ldots,\xi_{k+1})\setminus \{0\},
$$
therefore ${\rm ind}(u,\U)\geq k+1$, a contradiction.

We now prove (ii), following the ideas in \cite[Theorem 5.9]{GNY2004}. 
Given an open set $\mathcal O$ and a pair of sets $E,\,F \subset \mathcal O$,  the {\it $p$-capacity between $E$ and $F$ inside $\mathcal O$} for $p>1$ is defined as
 $${\rm Cap}_p(E,\,F,\,\mathcal O)=\inf\{\|v\|^p_{ W^{1,\,p} (\mathcal O)}: \ v\in\Delta(E,\,F )\},$$
 where $\Delta(E,\,F )$ denotes the class of all functions $v\in W^{1,\,p}(\mathcal O)$ that are continuous
in $\mathcal O$ and satisfy $v=1$ on $E$, and $v=0$ on $F$. 
In particular, if $A=B_\rho(z)\setminus B_{\rho/2}(z)$ is an annulus such that 
$2A:=B_{2\rho}(z)\setminus B_{\rho/4}(z)$ is contained inside $B_{2r}(\bar x)$, 
to control ${\rm Cap}_2(A,\partial (2A),B_{2r}(\bar x))$ we can choose the function $v_{z,\rho}(x):=\min\big\{1,(4\rho^{-1}|x-z|-1)_+, (2-\rho^{-1}|x-z|)_+\big\}$ to obtain
\begin{equation}\label{capa estimate}
{\rm Cap}_2(A,\,\partial(2A),\,B_{2r}(\bar x))\le \|v_{z,\rho}\|_{W^{1,2}} = C(n) |A|^{1-\frac 2 n}. 
\end{equation}
Let us now consider the metric space $X:=\overline{B_{2r}(\bar x)}$ endowed with the Euclidean metric. In this space, we call ``$X$-annuli'' sets of the form $\left(B_\rho(z)\setminus B_{\rho/2}(z)\right)\cap \overline{B_{2r}(\bar x)}$ for some $z \in \overline{B_{2r}(\bar x)}$. 

Define the measure on $X$ given by $\sigma:=\chi_{B_r(\bar x)}f'(u)dx$, and let $\kappa \in \mathbb N$ be a large constant to be fixed later.
Since $\sigma$ has no atoms,
we can apply \cite[Theorem 1.1]{GNY2004} to deduce the existence of 
a family of Euclidean annuli $\{A_i:=B_{\rho_i}(z_i)\setminus B_{\rho_i/2}(z_i)\}_{i=1}^{\kappa}$, with $z_i \in B_{2r}(\bar x)$, such that
\begin{equation}\label{Ai}
 \int_{B_r(\bar x)} f'(u)\, dx \le C_0(n)\,\kappa  \int_{B_{r}(\bar x)\cap A_i} f'(u)\, dx \qquad \forall\,i=1,\ldots,k+1
\end{equation}
and $\{(2 A_i)\cap B_{2r}(\bar x)\}_{i=1}^{\kappa}$ are pairwise disjoint.

With no loss of generality we can assume that $B_{r}(\bar x)\cap A_i \neq\emptyset$ (otherwise \eqref{Ai} would imply that $\int_{B_r(\bar x)} f'(u)\, dx=0$ and the result would be trivially true).
Let us split these annuli into two families: if $\rho_i<r/4$ then we say that $i \in \mathcal I_1$,
otherwise we say that $i \in \mathcal I_2$.

Note that, since $B_{r}(\bar x)\cap A_i \neq\emptyset$, for $i \in \mathcal I_2$ it holds
$(2 A_i)\cap B_{2r}(\bar x) \geq c_nr^n$ for some dimensional constant $c_n>0$.
Hence, since the sets $\{(2 A_i)\cap B_{2r}(\bar x)\}_{i\in \mathcal I_2}$ are disjoint,
we deduce that $\#\mathcal I_2 \leq N_n$ for some dimensional constant $N_n\geq 1$.

On the other hand, when $i \in \mathcal I_1$, since $B_{r}(\bar x)\cap A_i \neq\emptyset$ and $\rho_i<r/4$ it follows that $(2 A_i)\cap B_{2r}(\bar x)=2A_i$,
hence the sets $\{2 A_i\}_{i\in \mathcal I_1}$ are pairwise disjoint. 
Also, it follows from \eqref{capa estimate} that
\begin{equation}\label{eq:capa2}
{\rm Cap}_2(A_i,\,\partial(2A_i),\,B_{2r}(\bar x))\le C(n)  |A_i|^{1-\frac2n}.\end{equation}
Now, fix $\kappa:=N_n+2(k+1)$ so that $\#\mathcal I_1\geq 2(k+1)$.
Since the sets $\{2 A_i\}_{i\in \mathcal I_1}$ are pairwise disjoint and contained inside $B_{2r}(\bar x)$, there exists a subset of indices $\mathcal I_{1}'\subset \mathcal I_1$ such that $\#\mathcal I_1'\geq k+1$ and
$$
|2A_i|\leq \frac{1}{k+1}|B_{2r}| \qquad \forall\,i \in \mathcal I_1',
$$
that combined with \eqref{eq:capa2} gives
\begin{equation}\label{eq:vol}
{\rm Cap}_2(A_i,\,\partial(2A_i),\,B_{2r}(\bar x))\le C_1(n) (1+k)^{\frac{2}n -1}r^{n-2} \qquad \forall\,i \in \mathcal I_1'.
\end{equation}
Now,  assume by contradiction that \eqref{eq:L1 f'u} does not hold with  $C_n=4C_0C_1(N_n+1)$, namely
\begin{equation}
\label{eq:contradict}
\int_{B_r(\bar x)}f'(u) > C_n(1+k)^{\frac2n}r^{n-2},
\end{equation}
where $C_0$ and $C_1$ are as in \eqref{Ai} and \eqref{eq:vol}. Then,
since $2(N_n+1)(k+1)\geq \kappa$, combining   \eqref{eq:contradict}, \eqref{eq:vol}, and \eqref{Ai}, we get
$$
{\rm Cap}_2(A_i,\,\partial(2A_i),\,B_{2r}(\bar x))< (2C_0\kappa)^{-1}
\int_{B_r(\bar x)} f'(u)\, dx \le \frac 1 2 \int_{A_i} f'(u)\, dx\qquad \forall\,i \in \mathcal I_1'.
$$
Choose functions $\xi_i \in C^1_c(2A_i)$ that almost minimize the capacity ${\rm Cap}_2(A_i,\,\partial(2A_i),\,B_{2r}(\bar x))$, 
so that
$$\int_{B_{2r}(\bar x)} |\nabla\xi_i|^2\, dx \le \frac23 \int_{A_i} f'(u)\, dx \leq \frac 23 \int_{B_{2r}(\bar x)} f'(u) \xi_i^2\, dx \qquad \forall\,i \in \mathcal I_1'.$$
Since the sets $\{2A_i\}_{i\in \mathcal I_1'}$ are pairwise disjoint and $\# \mathcal I_1'\geq k+1$, we conclude that  $\{\xi_i\}_{i \in \mathcal I_1'}$ spans a $(k+1)$-dimensional subspace of $C^1_c(B_{2r}(\bar x))$ where the stability inequality fails.
This contradicts ${\rm ind}(u,B_{2r}(\bar x))\leq k$ and concludes the proof. 
\end{proof}

We now prove a crucial convergence result for weak $W^{1,2}$ limits of smooth solutions with bounded Morse index.
Note that, a consequence of Proposition~\ref{prop:stability} below, limit of smooth solutions with bounded Morse index are still smooth.
However the result does not provide any uniform bound on the sequence $u_j$. In particular, it could be that $\|u_j\|_{L^\infty} \to \infty$, as the example provided by \eqref{critical case} shows.

\begin{prop}\label{prop:stability}
Let $n\leq 9$, $\U\subset \R^n$ an open set, and $u_j \in C^2(\U)$ a sequence of functions satisfying 
$$
-\Delta u_j= f_j(u_j) \qquad\text{ in } \U
$$
with $f_j:\R\to \R$ nonnegative, increasing, and of class $C^1$.
Assume that 
$$
\text{${\rm ind}(u_j,\U)\leq k$ for some $k \in \mathbb N$,}\qquad\text{$\sup_{j}\|u_j\|_{W^{1,2}(\U)} <+\infty$,}\qquad \text{$f_j \to f_\infty$ in $C^1_{\rm loc}(\R)$.}
$$
Then there exist a subsequence $u_{j(m)}$ and a discrete set $\Sigma_\infty \subset \U$, with $\# \Sigma_\infty \leq k$, such that
$$
\text{$u_{j(m)}\rightharpoonup u_\infty$ in $W^{1,2}(\U)$},\qquad \text{$u_{j(m)} \to u_\infty$ in $C^2_{\rm loc}(\U\setminus \Sigma_\infty)$,}
$$
and $u_\infty$ satisfies 
$$
-\Delta u_\infty= f_\infty(u_\infty) \text{ in } \U,\qquad f_\infty'(u_\infty) \in L^1_{\rm loc}(\U), \qquad  {\rm ind}(u_\infty,\Omega)\leq k,\qquad u_\infty \in C^2(\U).
$$
\end{prop}

\begin{proof}
Given $x \in \U$,  for any $j$ we denote by $r_{j,x}$  the largest radius where $u_j$ is stable around $x$:
$$
r_{j,x}:=\sup\{r \in [0,\dist(x,\partial \U))\,:\,{\rm ind}(u_j,B_{r}(x))=0\},
$$
Then, we  define
$$
r_{\infty,x}:=\limsup_{j\to \infty}r_{j,x},\qquad \Sigma_\infty:=\{x \in \U\,:\,r_{\infty,x}=0\}.
$$
We claim that $\Sigma_\infty$ is a discrete set of cardinality at most $k$.

Indeed, suppose by contradiction that $\Sigma_\infty$ contains $k+1$ points $x_1,\ldots,x_{k+1}$, and fix
$$0<r<\min\left\{\min_{1\leq i \leq k+1}\dist(x_i,\partial \U)),\frac12 \min_{1\leq i,l\leq k+1}|x_i-x_l| \right\}.$$
Since $r_{j,x_i} \to 0$ as $j \to \infty$ (because $x_i \in \Sigma_\infty$), for $j$ large enough  $u_j$ is unstable inside each of the balls $\{B_r(x_i)\}_{i=1}^{k+1}$.
However, since these balls are disjoint (by the choice of $r$), Lemma~\ref{lem:Morse}(i) provides the desired contradiction.

Consider now a family of compact sets $\{K_\ell\}_{\ell \in \mathbb N}$ such that $\U\setminus \Sigma_\infty=\cup_\ell K_\ell,$
and for any $\ell$ consider the covering of $K_\ell$ given by $\{B_{r_{\infty/2,x}}(x)\}_{x \in K_\ell}$. By compactness, there exists a finite set of points $\{x_i\}_{i \in \mathcal I_\ell}\subset K_\ell$ such that $K_\ell\subset \cup_{i \in \mathcal I_\ell}B_{r_\infty/2,x_i}(x_i)$.
Note that, since each set $\mathcal I_\ell$ is finite, for each $\ell \in \mathbb N$ we can choose a subsequence $j_\ell(m)$ such that
$$
r_{\infty,x_i}=\lim_{m\to \infty}r_{j_\ell(m),x_i}\qquad \forall\, i \in \mathcal I_1\cup\ldots\cup \mathcal I_\ell.
$$
Then, by a diagonal argument we can find a subsequence $j(m)$, independent of $\ell$, such that
$$
r_{\infty,x_i}=\lim_{m\to \infty}r_{j(m),x_i}\qquad \forall\, i \in \cup_{\ell \in \mathbb N} \mathcal I_\ell.
$$
Since the functions $u_{j(m)}$ are uniformly bounded in $W^{1,2}(\U)$, up to extracting a further subsequence, there exists a weak limit in $W^{1,2}(\U)$ that we denote by $u_\infty$. 
We now want to show that $u_\infty$ satisfies all the desired properties.

First of all, for each $\ell \in \mathbb N$ we define the open set 
$$
\mathcal O_\ell:=\cup_{i \in \mathcal I_\ell}B_{r_\infty/2,x_i}(x_i) \supset K_\ell.
$$
Since  $u_{j(m)}$ is stable on $B_{r_{j(m)},x_i}(x_i)$ and $r_{j(m),x_i} \to r_{\infty,x_i}$ as $m\to \infty$, it follows by \cite[Theorem 1.2]{CFRS2020} and elliptic regularity\footnote{Recall that $n\leq 9$, and note that $f_j$ are uniformly $C^1$ on compact set since they converge to $f_\infty$.} that 
$$
\|u_{j(m)}\|_{C^{2,\alpha}(\mathcal O_\ell)}\leq C_{\ell,\alpha} \qquad \forall\,m \gg 1,\,\forall\,\alpha \in (0,1),
$$
which implies that $u_{j(m)} \to u_\infty$ in $C^2(\mathcal O_\ell)$. Since $\cup_{\ell} \mathcal O_\ell=\U\setminus \Sigma_\infty$, this proves the convergence in $C^2_{\rm loc}(\U\setminus \Sigma_\infty)$.

To show that $u_\infty$ solves the desired equation, by the $C^2_{\rm loc}(\U\setminus \Sigma_\infty)$ convergence it follows immediately that
$$
-\Delta u_\infty=f_\infty(u_\infty) \qquad \text{in }\U\setminus\Sigma_\infty. 
$$
Then, since $u_\infty \in W^{1,2}(\mathcal U)$ and $\Sigma_\infty$ consists of finitely many points (hence it has zero $W^{1,2}$-capacity), the equation $-\Delta u_\infty=f_\infty(u_\infty)$ must hold inside the whole domain $\U$.

We now note that, thanks to Lemma~\ref{lem:Morse}(ii), 
$$
\int_{B_r(\bar x)}f_{j(m)}'(u_{j(m)}) \leq C_n(1+k)^{\frac2n}r^{n-2} \qquad \forall\,B_{2r(\bar x)}\subset \U.
$$
Since $f_{j(m)}'(u_{j(m)})$ are nonnegative and converge  pointwise to $f_{\infty}'(u_{\infty})$ inside $\U \setminus \Sigma_\infty$ (and so a.e.), Fatou's Lemma implies that 
$$
\int_{B_r(\bar x)}f_\infty'(u_\infty) \leq C_n(1+k)^{\frac2n}r^{n-2} \qquad \forall\,B_{2r(\bar x)}\subset \U,
$$
thus $f'_\infty(u_\infty) \in L^1_{\rm loc}(\U)$.

Now, to prove the bound on the index, assume by contradiction that there exists a $k+1$ dimensional subspace $X'\subset C^1_c(\U)$ where 
$$
Q_\infty(\xi):=\int_{\U}\left(|\nabla \xi|^2-f_\infty'(u_\infty)\xi^2\right)\,dx<0 \qquad \forall\, \xi \in X'\setminus\{0\}.
$$
We claim that also $Q_{j(m)}$ is strictly negative on $X'\setminus \{0\}$ for $m$ sufficiently large.
Indeed, if not, by homogeneity there exists a sequence $\xi_m \in X'\setminus \{0\}$, with $\|\xi_m\|_{C^1}=1$, such that 
$Q_{j(m)}(\xi_m)\geq 0$. Since $X'\subset C^1_c(\U)$ is finite dimensional, all functions $\xi_m$ live in a fixed compact set and, up to a subsequence, they converge in $C^1_c(\U)$ to a limiting function $\xi_\infty \in X'$ with $\|\xi_\infty\|_{C^1}=1$.
In particular,
$$
\int_{\U}|\nabla \xi_m|^2\,dx \to \int_{\U}|\nabla \xi_\infty|^2\,dx\qquad \text{as }m\to \infty.
$$
Also, since $f_{j(m)}'(u_{j(m)})\xi_m^2$ are nonnegative and converge  pointwise to $f_{\infty}'(u_{\infty})\xi_\infty^2$ inside $\U \setminus \Sigma_\infty$ (and so a.e.), Fatou's Lemma implies that 
$$
\liminf_{m\to \infty}\int_{\U}f_{j(m)}'(u_{j(m)})\xi_m^2 \geq \int_{\U}f_{\infty}'(u_{\infty})\xi_\infty^2\,dx.
$$
Combining these two facts, we deduce that
$$
0\leq \limsup_{m\to \infty} {Q_{j(m)}(\xi_m)} \leq Q_{\infty}(\xi_\infty),
$$
a contradiction since $\xi_\infty \in X'\setminus \{0\}$.
Hence $Q_{j(m)}$ is strictly negative on $X'\setminus \{0\}$ for $m$ sufficiently large, which is impossible since ${\rm ind}(u_{j(m)},\U)\leq k$.
This contradiction proves that ${\rm ind}(u_{\infty},\U)\leq k$.

Finally, to prove that $u_\infty \in C^2(\U)$, we recall that finite Morse index solutions are locally stable (see for instance \cite[Proposition 1.5.1]{D2011} or \cite[Proposition 2.1]{DDF2011}).
Hence, we can apply Proposition~\ref{removable singularity} and elliptic regularity around each of the points in $\Sigma_\infty$ to deduce that $u_\infty \in C^2(\U)$.
\end{proof}

Our next goal is to show a uniform $W^{1,2}$ integrability estimate for  finite index solutions. It is for this result that the growth assumption on $f$ plays a crucial role. 
Before stating and proving it, we first recall the  following simple estimate that can be found, for instance, in \cite[Lemma A.1]{CFRS2020}.

\begin{lem}
\label{lem:Delta}
Let $f:\R \to \R$ be a nonnegative function, and let $v \in C^2$ solve $-\Delta v=f(v)$ inside $B_r(\bar x)$. Then
$$
\int_{B_{r/2}(\bar x)} f(v) \,dx \leq C_nr^{-2} \int_{B_{r}(\bar x)} |v|\,dx.
$$
\end{lem}
 We begin by proving a uniform $W^{1,2}$ integrability estimate for stable solutions, that will be used below to address the general case.
 
\begin{prop}\label{W12 decay}
Let $f \in C^1$ be nonnegative, and let $u\in C^2$  be a nonnegative stable solution to $-\Delta u=f(u)$ in $B_{r}(\bar x)$ for some $r \in (0,1]$. Assume that $f$ satisfies \eqref{supercritical} for some $c_1>0$.  Then there exists  $\delta=\delta(n, \ez)>0$ such that 	\begin{equation}
	\label{no atom}
\int_{B_\rho(\bar x)} |\nabla u|^2\, dx\le C(c_1,n) \,\rho^\delta \qquad \text{for all }0<\rho<\frac r 4.
	\end{equation}
\end{prop}

\begin{proof}
With no loss of generality we can assume $\bar x=0$.

By H\"older inequality, \eqref{supercritical}, and Lemma~\ref{lem:Delta},
for any ball $B_{2\rho}(z)\subset B_r$ it holds
\begin{multline}
\label{eq:L1 bound}
	\bigg(  \rho^{-n} \int_{B_{\rho}(z)} u\, dx\bigg)^{\frac {n+2}{n-2}+\ez} \leq C(n) \rho^{-n} \int_{B_{\rho}(z)} u^{\frac {n+2}{n-2}+\ez}\, dx\\
	 \le C(n,c_1)   \rho^{-n} \int_{B_{\rho}(z)} f(u)\, dx 
	 \le C_0 \rho^{-2-n} \int_{B_{2\rho}(z)}  u \,dx,
\end{multline}
where $C_0=C_0(n,c_1)$.
Let $\delta=\delta(n,\ez)>0$ be small enough so that
\begin{equation}
\label{eq:delta}
\biggl(\frac{n-2 } {2}-\delta\biggr)\biggl(\frac {n+2}{n-2}+\ez\biggr)-2\geq \frac{n-2 } {2}-\delta,
\end{equation}
and define
$$
G(z,\rho):=\max\left\{1,\gamma\sup_{B_s(y)\subset B_\rho(z)}s^{-\frac{n+2 } {2}-\delta}\int_{B_{s}(y)} u\,dx\right\},
$$
where $\gamma \in (0,1)$ is a small constant to be fixed later.
Then, thanks to \eqref{eq:L1 bound} and \eqref{eq:delta}, whenever $B_{2\rho}(z)\subset B_r$ we have
\begin{multline}
\label{eq:G}
G(z,\rho)\leq G(z,\rho)^{\frac {n+2}{n-2}+\ez}\leq 1+ \gamma^{\frac {n+2}{n-2}+\ez}\sup_{B_s(y)\subset B_\rho(z)}\bigg(s^{-\frac{n+2 } {2}-\delta}\int_{B_{s}(y)} u\,dx\bigg)^{\frac {n+2}{n-2}+\ez}\\
\leq 1+ C_0\gamma^{\frac {n+2}{n-2}+\ez}\sup_{B_s(y)\subset B_\rho(z)}s^{-\frac{n+2 } {2}-\delta}\int_{B_{2s}(y)} u\,dx\leq 1+ C_0 \gamma^{\frac {4}{n-2}+\ez} G(z,2\rho).
\end{multline}
We now claim that $G(0,r/2)$ is uniformly bounded.

To show this, consider the quantity
$$
Q:=\sup_{z \in B_{r},\\ \rho \leq r-|z|} G(z,\rho/2).
$$
Note that, since $u$ is of class $C^2$, $Q$ is a finite constant. Also, we can assume that $Q>2$ (otherwise there is nothing to prove). 
Consider now $z \in B_{r}$ and $\rho \leq r-|z|$ such that $G(z,\rho/2)\geq Q/2$. Since $Q/2>1$, it follows from the definition of $G$ that there exists $B_s(y)\subset B_{\rho/2}(z)$ such that $\gamma s^{-\frac{n+2 } {2}-\delta}\int_{B_{s}(y)} u \geq Q/3$.
We can now cover $B_s(y)$ with $N_n$ balls $\{B_{s/4}(y_k)\}_{k=1}^{N_n}$ with $y_k \in B_s(y)\subset B_{\rho/2}(z)$, where $N_n$ is a dimensional constant, and observe that
\begin{multline}
\label{eq:Q yk}
\frac{Q}3 \leq \gamma s^{-\frac{n+2 } {2}-\delta}\int_{B_{s}(y)} u \,dx\leq  \gamma s^{-\frac{n+2 } {2}-\delta}\sum_{k=1}^{N_n}\int_{B_{s/4}(y_k)}u\,dx\\
 \leq \gamma(s/4)^{-\frac{n+2 } {2}-\delta}\sum_{k=1}^{N_n}\int_{B_{s/4}(y_k)}u\,dx\leq \sum_{k=1}^N G(y_k,s/4).
\end{multline}
Note now that, since $s \leq \rho/2$ and $y_k \in B_{\rho/2}(z)$,
$$
|y_k|+s \leq |z|+\frac{\rho}2+\frac{\rho}2 \leq |z|+\rho\leq r.
$$
In particular $B_{s/2}(y_k)\subset B_r$, and it follows by \eqref{eq:G} and the definition of $Q$ that
$$
G(y_k,s/4)\leq 1+C_0 \gamma^{\frac {4}{n-2}+\ez} G(y_k,s/2) \leq 1+C_0 \gamma^{\frac {4}{n-2}+\ez}Q.
$$
Combining this bound with \eqref{eq:Q yk}, this yields
$$
\frac{Q}3 \leq N_n\left(1+C_0 \gamma^{\frac {4}{n-2}+\ez}Q\right) \leq N_n\left(1+C_0 \gamma^{\frac {4}{n-2}}Q\right),
$$
and by choosing $\gamma$ small enough (depending only on $C_0$ and the dimension), we conclude that $Q \leq 4N_n$, and therefore
\begin{equation}
\label{eq:sup bound}
\gamma\sup_{s \leq r/2}s^{-\frac{n+2 } {2}-\delta}\int_{B_{s}(0)} u\,dx \leq G(0,r/2)\leq Q \leq 4N_n,
\end{equation}
as desired.

Recall now that, by \cite[Theorem 1.2]{CFRS2020}, if $u$ is stable on a ball $B$ then
$$
\|u\|_{W^{1,2}\left(\frac 1 2 B\right)}\le C(n)\, (\diam(B))^{-\frac{n+2}{2}}\|u\|_{L^1(B)}.
$$
Combining this estimate with \eqref{eq:sup bound},  we obtain \eqref{no atom}.\end{proof}

We next improve this result to solutions with finite Morse index.
\begin{prop}\label{uniform integrable}
Let $f \in C^1$ be nonnegative, and let $u\in C^2$  be a nonnegative solution to $-\Delta u=f(u)$ in $B_{r}(\bar x)$ for some $r>0$. 
Assume that ${\rm ind}(u,B_r(\bar x))\leq k$ for some $k \in \mathbb N$, and that $f$ satisfies \eqref{supercritical} for some $c_1>0$. 
Then 
$$
\int_{B_\rho(\bar x) } |\nabla u|^2\, dx\le C(c_1,n,\epsilon)\,k\, \rho^\delta\qquad \text{for all }\,\rho \in (0,r/4),
$$
where $\delta=\delta(n,\epsilon)>0$ is as in Proposition~\ref{W12 decay}.
\end{prop}
\begin{proof}
Let $M >1$ be a fixed constant\footnote{One can choose $M$ to be any constant larger than $1$, for instance $M=2$. However, for notational convenience we prefer to use the notation $M$ instead of fixing its value, as we believe that the estimates become easier to follow.}, define the set $Q^0:=B_\rho(\bar x)$, and consider the covering of $Q^0$ given by 
$\left\{B_{M^{-1} \rho}(z)\right\}_{z\in Q^0}.$
By Besicovitch Covering Theorem, 
there exist a dimensional constant $N_n \in \mathbb N$ and a subfamily of balls $\left\{B^0_\ell\right\}_{\ell \in \mathcal I_0}\subset \left\{B_{ M^{-1} \rho}(z)\right\}_{z\in Q^0}$
such that 
 \begin{equation}
 \label{eq:finite overlap 1}
 1\le\sum_{\ell \in \mathcal I_0} \chi_{B^0_\ell}(y) \le N_n\qquad \text{for all }y \in Q^0.
\end{equation}
 In particular, since these balls have radius $M^{-1} \rho$ and are contained inside $B_{2\rho}(\bar x)$, it follows that
$$
 \# \mathcal I_0\,|B_{M^{-1} \rho}| \leq N_n |B_{2\rho}| \qquad \Rightarrow \qquad
  \# \mathcal I_0 \leq 2^n M^n N_n.
$$
 Moreover, since each point $y \in B_{\rho}(\bar x)$ is covered by at most $N_n$ balls of radius $M^{-1}\rho$, then the same is true if we double the radius of the balls: more precisely, there exists a dimensional constant $N'_n\in \mathbb N$ such that\footnote{A simple way to see this is to note that, as a consequence of \eqref{eq:finite overlap 1}, we can split $\left\{B^0_\ell\right\}_{\ell \in \mathcal I_0}$ into $N_n$ subfamilies of balls, where the balls of each subfamily are disjoint. This implies that the centers of the balls of each subfamily are at mutual distance at least $2M^{-1}\rho$. Then, if we double the radius, the overlapping for each of these subfamilies is bounded by a dimensional constant $C_n\geq 1$.}
 \begin{equation}
 \label{eq:finite overlap 1 2}
 1\le\sum_{\ell \in \mathcal I_0} \chi_{2B^0_\ell} (y)\le N'_n\qquad \forall\,y \in Q^0.
\end{equation}
Let us split $\left\{2B^0_\ell\right\}_{\ell \in \mathcal I_0}$ into $N'_n$ subfamilies of balls, where the balls of each subfamily are disjoint. 
As ${\rm ind}(u_s)\le k$, we can apply Lemma~\ref{lem:Morse}(i) to each subfamily. Then we deduce that, except for at most $N'_nk$ balls, say $B^0_1,\ldots,\,B^0_{k_0}$ with $k_0 \le N_n'k$, the function $u$ is stable inside each ball $\{2 B^0_\ell\}_{\ell \in\mathcal I_0 \setminus \{1,\ldots, k_0\}}$. Thus by Proposition~\ref{W12 decay}, we have
$$\int_{B^0_\ell} |\nabla u|^2\, dx \le C(c_1,n) M^{-\delta}  \rho^\delta
\qquad \forall\,\mathcal I_0 \setminus \{1,\ldots, k_0\}.
$$
Now, we consider the set
$Q^1:=\bigcup_{1\le \ell\le k_0} B^0_{\ell}$
and the covering $\left\{B_{M^{-2} \rho}(z)\right\}_{z\in Q^1}.$
Again by Besicovitch Covering Theorem, there exists a subfamily of balls $\left\{B^1_\ell\right\}_{\ell \in \mathcal I_1}\subset \left\{B_{ M^{-2} \rho}(z)\right\}_{z\in Q^1}$
such that 
 \begin{equation}
 \label{eq:finite overlap 2}
 1\le\sum_{\ell \in \mathcal I_1} \chi_{B^1_\ell}(y) \le N_n\qquad \text{for all }y \in Q^1.
\end{equation}
Also, since these balls are contained inside $\bigcup_{1\le \ell\le k_0} 2B^0_{\ell}$,
it follows that (recall that $k_0 \leq N_n'k$)
$$
 \# \mathcal I_1\,|B_{M^{-2} \rho}| \leq k_0 N_n |B_{2M^{-1} \rho}| \qquad \Rightarrow \qquad
  \# \mathcal I_1 \leq 2^n M^n k_0 N_n \leq 2^n M^n N_n' N_n k.
$$
Furthermore, as before, 
$$
 1\le\sum_{\ell \in \mathcal I_1} \chi_{2B^1_\ell} (y)\le N'_n\qquad \forall\,y \in Q^1.
 $$
Hence (up to renaming the indices) $u$ is stable inside each ball $\{2 B^1_\ell\}_{\ell \in \mathcal I_1 \setminus \{1,\ldots, k_1\}}$ with $k_1\le N_n'k$, and therefore
$$
\int_{B^1_\ell} |\nabla u|^2\, dx \le C(c_1,n) M^{-2\delta}  \rho^\delta
\qquad \forall\,\ell \in \mathcal I_1 \setminus \{1,\ldots, k_1\}.
$$ 
To continue this construction, define 
$$Q^2:=\bigcup_{1\le \ell\le k_2} B^2_{\ell}.$$
Then, we can apply the very same argument used for $Q^1$ to find a family of balls $\{B^2_\ell\}_{\ell \in \mathcal I_2}$, with $\#\mathcal I_2 \leq 2^n M^n k_1 N_n\leq 2^n M^n N_n' N_n k$, such that
$$
\int_{B^2_\ell} |\nabla u|^2\, dx \le C(c_1,n) M^{-3\delta}  \rho^\delta
\qquad \forall\,\ell \in \mathcal I_2 \setminus \{1,\ldots, k_2\},\qquad \text{with }k_2 \leq N_n'k.
$$ 
Iterating this construction, we obtain that the family of balls
$\{B^j_\ell\}_{\ell \in \mathcal I_j \setminus \{1,\ldots,k_j\},\,j \in \mathbb N}$ covers
$Q^0\setminus K$, with $K:=\cap_{j \in \mathbb N}Q^j$,\footnote{Here one could note that, since $u$ is smooth, every ball sufficiently small is stable and therefore $Q^j$ is empty for $j$ large enough, hence $K=\emptyset$. However this information is not needed, and this proof also applies to weak solutions with bounded index.}
and 
$$
\int_{B^j_\ell} |\nabla u|^2\, dx \le C(c_1,n) M^{-j\delta}  \rho^\delta
\quad \forall\,\ell \in \mathcal I_j \setminus \{1,\ldots, k_j\},\qquad \#\mathcal I_j \leq 2^n M^n N_n' N_n k,\qquad k_j \leq N_n'k.
$$
Since $K$ has measure zero (because $|Q^j|\leq k_j|B_{M^{-j}\rho}| \leq N_n' k |B_{M^{-j}\rho}| \to 0$ as $j\to\infty$), we have
\begin{multline*}
\int_{Q^0} |\nabla u|^2\, dx=\int_{Q^0\setminus K} |\nabla u|^2\, dx
 \le  \sum_{j=0}^\infty \sum_{\ell \in \mathcal I_j \setminus \{1,\ldots,k_j\}} \int_{B^j_\ell} |\nabla u|^2\, dx \\
\le C(c_1,n) \biggl(\sum_{j=0}^\infty \# \mathcal I_j  M^{-j\delta}\biggr)  \rho^\delta 
  \le C(c_1,n,\delta)\,k\,\rho^\delta.
\end{multline*} 
Recalling that $\delta=\delta(n,\epsilon),$ this concludes the lemma.
\end{proof}

The next result is a powerful $\varepsilon$-regularity theorem which shows the following: given $\gamma\in (0,1),$ if the $W^{1,2}$ norm of a solution in a ball $B_r$ decays like $r^{n-2+2\gamma}$ for $r \in [\varepsilon,1]$ with $\varepsilon$ small enough, then it decays all the way to the origin.

\begin{prop}
\label{prop:eps reg}
Let $g:\R\to \R$ be a continuous nonnegative function, let $u \in C^2(B_1)$ solve 
$-\Delta u=f(u)$ for some increasing function $f:\R\to \R$ of class $C^1$ satisfying $ 0\leq f \leq g$ and \eqref{F and f},
and assume that ${\rm ind}(u,B_1)\leq k$ and $\int_{B_1}|u|\leq M$. Then, for any $\gamma \in (0,1)$ there exists $m_0=m_0(n,g,\epsilon,t_0,k,M,\gamma) \in \mathbb N$ such that the following holds:\\
Suppose that
$$
\int_{B_r}|\nabla u|^2\,dx \leq r^{n-2+2\gamma} \qquad \forall\,r \in [2^{-m_0},1].
$$
Then
\begin{equation}
\label{eq:decay to 0}
\int_{B_r}|\nabla u|^2\,dx \leq r^{n-2+2\gamma} \qquad \forall\,r \in [0,1]
\end{equation}
and $|u(0)|\leq M+C(n,\gamma)$.
\end{prop}
\begin{proof}
We begin with the proof of \eqref{eq:decay to 0}.
For that, it suffices to prove the following implication: if
$$
\int_{B_r}|\nabla u|^2\,dx \leq r^{n-2+2\gamma} \qquad \forall\,r \in [2^{-m},1]
$$
for some $m \geq m_0$, then
$$
\int_{B_r}|\nabla u|^2\,dx \leq r^{n-2+2\gamma} \qquad \forall\,r \in [2^{-(m+1)},1].
$$
Indeed, iterating this result with $m=m_0,m_0+1,\ldots$, the result follows.

To prove the implication above, we argue by contradiction. If it was false, we could find a sequence of functions $u_j \in C^2(B_1)$, and $f_j \in C^1(\R)$ satisfying \eqref{F and f}, such that
$$
-\Delta u_j=f_j(u_j),\qquad \text{$f_j$ increasing},\qquad 0 \leq f_j \leq g,\qquad {\rm ind}(u_j,B_1)\leq k, \qquad \int_{B_1}|u_j|\leq M,
$$
and a sequence $m_j \to \infty$, such that
\begin{equation}
\label{eq:decay m}
\int_{B_r}|\nabla u_j|^2\,dx \leq r^{n-2+2\gamma} \qquad \forall\,r \in [2^{-m_j},1]
\end{equation}
but
\begin{equation}
\label{eq:decay m+1 fails}
\int_{B_{r_j}}|\nabla u_j|^2\,dx \geq r^{n-2+2\gamma} \qquad \text{for some }r_j \in [2^{-(m_j+1)},2^{-m_j}].
\end{equation}
We introduce the notation $A_r:=B_{r}\setminus B_{r/2}$.

We first note that, as a consequence of \eqref{eq:decay m} and the bound $\int_{B_1}|u_j|\leq M$, it follows that 
\begin{equation}
\label{eq:bound L1}
\bint_{A_{2^{-(m_j+1)}}}|u_j|\,dx \leq M+C(n,\gamma).
\end{equation}
Indeed, thanks to \eqref{eq:decay m}, for any $2^{-m_j}\leq s\leq r \leq 1$ we have
\begin{equation}
\label{eq:L1 annuli}
\begin{split}
\bigg| \bint_{\partial B_r}|u_j| - \bint_{\partial B_s}|u_j|\biggr|&\leq 
\bint_{\partial B_1}|u_j(ry) - u_j(sy)|\,dy \leq \int_s^r \biggl(\bint_{\partial B_1}|\nabla u_j(\tau y)|\,dy\biggr)\,d\tau\\
&=\int_{B_{r}\setminus B_s}\frac{|\nabla u_j(x)|}{|x|^{n-1}}\,dx\leq \sum_{\ell=1}^{m_j}
\int_{A_{2^{-\ell}}} \frac{|\nabla u_j(x)|}{|x|^{n-1}}\,dx\leq \sum_{\ell=1}^{m_j}
2^{\ell(n-1)}\int_{A_{2^{-\ell}}} |\nabla u_j|\,dx \\
&=C(n)\sum_{\ell=1}^{m_j}
2^{-\ell}\bint_{A_{2^{-\ell}}} |\nabla u_j|\,dx\leq C(n)\sum_{\ell=1}^{m_j}
2^{-\ell}\biggl(\bint_{A_{2^{-\ell}}} |\nabla u_j|^2\,dx\biggr)^{1/2}\\
&\leq C(n)\sum_{\ell=1}^{m_j}
2^{-\ell}2^{-\ell(\gamma -1)}=C(n)\sum_{\ell=1}^{m_j}2^{-\ell\gamma}\leq C(n,\gamma),
\end{split}
\end{equation}
therefore
$$
\biggl| \bint_{A_1}|u_j|\,dx-\bint_{A_{2^{-(m_j+1)}}}|u_j|\,dx\biggr| \leq C(n,\gamma),
$$
and \eqref{eq:bound L1} follows.

To simplify the notation, we set $r_j:=2^{-m_j}$, and we define
$$
a_j:=\bint_{A_{2r_j}}u_j\,dx,\qquad w_j(x):=r_j^{-\gamma}[u_j(r_jx)-a_j],
$$
so that 
\begin{equation}
	\label{eq:wj}
	\bint_{A_2}w_j=0,\qquad -\Delta w_j=h_j(w_j),\qquad h_j(t):=r_j^{2-\gamma}f_j(a_j+r_j^\gamma t).
\end{equation}
Note that ${\rm ind}(w_j,B_{2^{m_j}})\leq k$ and $0 \leq h_j \leq r_j^{2-\gamma}g(a_j+r_j^\gamma t),$ so it follows from \eqref{eq:bound L1} that $h_j \to 0$ in $C^1_{\rm loc}$.
Also \eqref{eq:decay m} and \eqref{eq:decay m+1 fails} imply that
\begin{equation}
	\label{eq:growth wj}
\int_{B_{2^\ell}}|\nabla w_j|^2 dx \leq 2^{\ell(n-2+2 \gamma)} \qquad \forall\,0 \leq \ell \leq m_j
\end{equation}
and
\begin{equation}
	\label{eq:lower wj}
\int_{B_{1}}|\nabla w_j|^2 dx \geq 2^{-(n-2+2 \gamma)}.
\end{equation}
Thus, thanks to Proposition~\ref{prop:stability} and a diagonal argument we deduce that, up to a subsequence,
$$
\text{$w_{j}\rightharpoonup w_\infty$ in $W^{1,2}_{\rm loc}(\R^n)$},\qquad \text{$w_{j} \to w_\infty$ in $C^2_{\rm loc}(\R^n\setminus \Sigma_\infty)$,}
$$
where $\Sigma_\infty$ has cardinality at most $k$, and $w_\infty$ satisfies (by \eqref{eq:growth wj} and \eqref{eq:wj})
$$
-\Delta w_\infty= 0 \quad \text{in } \R^n,\qquad \bint_{A_2}w_\infty=0,\qquad \int_{B_{2^\ell}}|\nabla w_\infty|^2 dx \leq 2^{\ell(n-2+2 \gamma)} \qquad \forall\,\ell \geq 0.
$$
Then it follows from Liouville Theorem for harmonic functions that $w_\infty\equiv 0$,\footnote{Indeed, by Liouville Theorem $w_\infty$ must be a harmonic polynomial, and the  bound $\int_{B_{2^\ell}}|\nabla w_\infty|^2 dx \leq 2^{\ell(n-2+2 \gamma)}$ for $\ell \geq 0$ implies that $w_\infty$ must be constant (recall that $\gamma<1$). Finally, since $\int_{A_2}w_\infty=0$ we deduce that $w_\infty \equiv 0$.} and therefore
$$
\text{$w_{j}\rightharpoonup 0$ in $W^{1,2}_{\rm loc}(\R^n)$},\qquad \text{$w_{j} \to 0$ in $C^2_{\rm loc}(\R^n\setminus \Sigma_\infty)$,}\qquad \#\Sigma_\infty \leq k.
$$
We now want to get a contradiction with \eqref{eq:lower wj}.

Consider the annuli
$$
B_2\setminus B_1,\qquad B_3\setminus B_2,\qquad \ldots,\qquad B_{k+2}\setminus B_{k+1}.
$$
Since $\#\Sigma_\infty \leq k$, there exists $\hat i \in \{1,\ldots,k+1\}$ such that $(B_{\hat i+1}\setminus B_{\hat i})\cap \Sigma_\infty =\emptyset.$ In particular, if we fix $\varphi\in C^\infty_c(B_{\hat i+3/4})$ nonnegative such that $\varphi|_{B_{\hat i+1/4}}=1$, then
$w_j \to 0$ in $C^2$ on $\{\nabla \varphi\neq 0\}.$

Now we  first test the equation for $w_j$ (see \eqref{eq:wj}) with $w_j\varphi$ to get 
\begin{equation}
 \label{ws f}
\begin{split}
r_j^{2-\gamma}  \int_{B_{\hat i+1}}f_j\bigl(a_j+r_j^{\gamma}   w_j\bigr) w_j\varphi \, dx
&= \int_{B_{\hat i+1}} -w_j\Delta w_j  \varphi\, dx  \\
&=\int_{B_{\hat i+1}}|\nabla w_j|^2\varphi +w_j\nabla w_j\cdot \nabla\varphi  \, dx \\&=\int_{B_{\hat i+1}}|\nabla w_j|^2\varphi \, dx +o(1), 
\end{split}
\end{equation}
where $o(1)$ denotes a quantity that goes to $0$ as $j\to \infty$, and the last equality follows from the $C^2$ convergence of $w_j$ to $0$ on the set $\{\nabla \varphi\neq 0\}.$

Similarly, testing \eqref{eq:wj} with
$(\nabla w_j\cdot x)\varphi$, we  obtain
\begin{align*}
\int_{B_{\hat i+1}}  r_j^{2-\gamma}  f_j\bigl(a_j+r_j^{\gamma}   w_j\bigr) (\nabla w_j\cdot x) \varphi \, dx
&=\int_{B_{\hat i+1}}-\Delta w_j (\nabla w_j\cdot x) \varphi\, dx\\
&=\int_{B_{\hat i+1}} D^2 w_j \nabla w_j\cdot x\varphi+ |\nabla w_j|^2\varphi +(x\cdot \nabla w_j) \nabla w_j\cdot \nabla\varphi \, dx\\
&=\left(1-\frac n 2\right)\int_{B_{\hat i+1}} |\nabla w_j|^2\varphi\, dx -\frac 1 2 \int_{B_{\hat i+1}} |\nabla w_j|^2 \nabla\varphi\cdot x\, dx+o(1)\\
&= \left(1-\frac n 2\right)\int_{B_{\hat i+1}} |\nabla w_j|^2\varphi\, dx+o(1).
\end{align*}
Also, if we define $F_j(t):=\int_{0}^t f_j(\tau)\, d\tau,$ then we can rewrite the first term above as follows:
\begin{align*}
&\int_{B_{\hat i+1}} r_j^{2-\gamma}   f_j\left(a_j+r_j^{\gamma}    w_j\right) (\nabla w_j\cdot x) \varphi\, dx\\
&= \int_{B_{\hat i+1}} r_j^{2- 2\gamma } \nabla \left[F_j\left(a_j+r_j^{\gamma}    w_j\right)-F_j(a_j)\right]  \cdot x \,\varphi \, dx\\
&=-n\int_{B_{\hat i+1}} r_j^{2-2 \gamma } \left[ F_j\left(a_j+r_j^{\gamma}    w_j\right)  -F_j(a_j)\right]  \varphi \, dx +\int_{B_{\hat i+1}} r_j^{2- \gamma }  \left[ F_j\left(a_j+r_j^{\gamma}    w_j\right)  -F_j(a_j)\right] x\cdot \nabla\varphi\, dx\\
&=-n\int_{B_{\hat i+1}} r_j^{2-2 \gamma } \left[ F_j\left(a_j+r_j^{\gamma}    w_j\right)  -F_j(a_j)\right]  \varphi \, dx +o(1),
\end{align*}
and we eventually get
\begin{equation}\label{ws F}
 r_j^{2- 2\gamma }\int_{B_{\hat i+1}} \left[ F_j\left(a_j+r_j^{\gamma}    w_j\right)  -F_j(a_j)\right]  \varphi \, dx=\frac{n-2}{2n} \int_{B_{\hat i+1}} |\nabla w_j|^2\varphi\, dx+o(1). 
\end{equation}
Now, given a constant $N>0$, we define the set
$$S_N:=\left\{x\in B_{\hat i+1}\colon r_j^{\gamma}  |w_j(x)|\le N\right\}.$$
Since $w_j$ is uniformly bounded in $W^{1,\,2}(B_{\hat i+1})$ and $a_j$ is uniformly bounded,
for any $N>0$ fixed we have 
 $$r_j^{2- \gamma }\int_{ S_N} \left[ F_j\left(a_j+r_j^{\gamma}    w_j\right)  -F_j(a_j)\right]  \varphi \, dx\to 0,\qquad r_j^{2- \frac \gamma 2 }\int_{S_N}   f_j\left(a_j+r_j^{\gamma}    w_j\right)  w_j \varphi \, dx\to 0,$$
 so it follows from \eqref{ws f} and \eqref{ws F} that
\begin{equation}\label{ws f2}
r_j^{2-\gamma}  \int_{B_{\hat i+1}\setminus S_N} f_j\bigl(a_j+r_j^{\gamma}   w_j\bigr) w_j\varphi \, dx
= \int_{B_{\hat i+1}}|\nabla w_j|^2\varphi \, dx+o(1)
\end{equation}
and
\begin{equation}\label{ws F2}
 r_j^{2- 2\gamma }\int_{B_{\hat i+1}\setminus S_N} \left[ F_j\left(a_j+r_j^{\gamma}    w_j\right)  -F_j(a_j)\right]  \varphi \, dx=\frac{n-2}{2n} \int_{B_{\hat i+1}} |\nabla w_j|^2\varphi\, dx+o(1). 
\end{equation}
Note now that, thanks to \eqref{F and f}, the fact that $a_j$ is uniformly bounded (see \eqref{eq:bound L1}), and that $0 \leq f_j \leq g$, there exists a large constant $N=N(M,n,\gamma,t_0,\ez)$ such that, for all $j$,
$$\left(\frac{2n}{n-2} +\frac \ez 2\right)\bigl[F_j(t+a_j)-F_j(a_j)\bigr]\leq  f_j(t+a_j) t \qquad \forall\, t \geq N.$$
Combining this inequality with \eqref{ws f2} and \eqref{ws F2}, we get
$$
\left(\frac{2n}{n-2} +\frac \ez 2\right)\frac{n-2}{2n}\int_{B_{\hat i+1}}|\nabla w_j|^2\varphi\leq \int_{B_{\hat i+1}}|\nabla w_j|^2\varphi
+o(1),
$$
or equivalently
$$
\frac{(n-2)\ez}{4n}\int_{B_{\hat i+1}}|\nabla w_j|^2\varphi\leq o(1).
$$
This contradicts \eqref{eq:lower wj} and concludes the proof of \eqref{eq:decay to 0}.

Now, to prove that bound on $|u(0)|,$ we observe that \eqref{eq:decay to 0} allows us to deduce the validity of
\eqref{eq:L1 annuli} for all $0\leq s \leq r \leq 1.$
In particular this implies that 
$$
\biggl| \bint_{A_1}|u_j|\,dx-|u(0)|\biggr| \leq C(n,\gamma),
$$
so $|u(0)|\leq M+C(n,\gamma)$ as desired.
\end{proof}

Finally, we conclude this section with a useful consequence of the moving plane method.

\begin{lem}\label{stable outside K}
Let $\Omega\subset \R^n$ be a bounded convex domain, and let $u \in C^2(\Omega)$ solve 
\eqref{equ} for some increasing positive function $f:\R\to \R$ of class $C^1$.
Then  there exists $\rho_0=\rho_0(\Omega) \in (0,1)$ such that 
$$
\max_{\Omega}u=\max_{\Omega_0}u,
$$
where
$\Omega_0:=\left\{x\in \Omega \colon \dist(x,\,\partial \Omega)> \rho_0 \right\}.$
\end{lem}

\begin{proof}
Recall that, since $f \geq 0$, the maximum principle implies that $u >0$ (unless $u\equiv 0$, in which case the result is trivially true).
Then, since $\Omega$ is bounded and convex, the result follows by the classical moving plane method (see also the footnote inside \cite[Proof of Corollary 1.4]{CFRS2020} for more details).
\end{proof}

\section{Uniform finite Morse index:  Proof of Theorem~\ref{mainthm 1}}

Let us assume, by contradiction, that there exists a sequence of $C^2$ solutions $u_j$ 
$$
\left\{
\begin{array}{ll}
-\Delta u_j= \lambda_j f_j(u_j) &\text{ in } \Omega,\\
 u_j=0 &\text{ on } \partial\Omega,
 \end{array}
 \right.
$$
where ${\rm ind}(u_j,\Omega)\leq k$, the functions $f_j:\R\to \R$ satisfy all the 
assumptions in the statement of the theorem, $0\leq \lambda_j\leq \hat\lambda,$ but $\|u_j\|_{L^\infty(\Omega)}\to \infty$ as $j\to \infty$. 
Since $f_j \in \mathcal K$ which is a compact family, up to a subsequence $f_j \to f_\infty$ in $C^1_{\rm loc}(\R)$
and $\lambda_j \to \lambda_\infty \in [0,\hat\lambda]$.
Define
\begin{equation}
\label{eq:def g}
\hat{f}(t):=\sup_{j}f_j(t)\qquad \forall\,t \in \R,
\end{equation}
so that $0 \leq f_j \leq \hat{f}$ for all $j$. Since the functions $f_j$ are locally uniformly Lipschitz (by the $C^1_{\rm loc}$ compactness), $\hat{f}$ is a continuous function.

We distinguish two cases, depending on the value of $\lambda_\infty$.

\subsection{The case $\lambda_\infty>0$}
Since $\lambda_j \to \lambda_\infty$, it follows from Remark~\ref{rmk:superlin} that, for $j$ large enough,
\begin{equation}\label{supercritical j}
  \lambda_j f_j(t)\geq c_1 t^{\frac{n+2}{n-2}+\ez}\qquad \forall\,t \geq 0,\qquad \text{for some }c_1>0.
  \end{equation}
Thanks to this bound, it follows from \cite[Proposition B.1]{CFRS2020} that
\begin{equation}
\label{eq:bound L1 j}
\|u_j\|_{L^1(\Omega)}\leq C_0=C_0(c_1,\Omega).
\end{equation}
Also, if we define $\Omega_\tau:=\left\{x\in \Omega \colon \dist(x,\,\partial \Omega)> \tau\right\}$, then  \eqref{supercritical j} and Proposition~\ref{uniform integrable} yield
$$
\|\nabla u_j\|_{L^{2}(\Omega_\tau)}\leq C_1=C_1(c_1,\Omega,\rho) \qquad \forall\,\tau>0.
$$
Since $\tau>0$ is arbitrary, Proposition~\ref{prop:stability} and a diagonal argument imply that, up to a subsequence,
\begin{equation}\label{W12 C2 j}
\text{$u_{j}\rightharpoonup u_\infty$ in $W_{\rm loc }^{1,2}(\U)$},\qquad \text{$u_{j} \to u_\infty$ in $C^2_{\rm loc}(\Omega\setminus \Sigma_\infty)$,}
\end{equation}
for some discrete set $\Sigma_\infty \subset \Omega$ with $\# \Sigma_\infty \leq k$, and some function $u_\infty \in C^2(\Omega)$.

Let $\rho_0 \in (0,1)$ and $\Omega_0$ be given by Lemma~\ref{stable outside K},
and define 
$$
\Sigma_\infty^0:=\Omega_0\cap \Sigma_\infty=\{\hat x_1,\ldots,\hat x_\ell\} \quad (\ell \leq k),\qquad
r_0:=\frac12\min\left\{\rho_0, \min_{1\leq i,l\leq \ell }|x_i-x_l|\right\}.
$$
Then it follows from \eqref{W12 C2 j} that, for any $\rho  \in (0,r_0)$,
$$
\max_{1\leq i \leq \ell} \|u_j -u_0\|_{C^2(B_{r_0}(\hat x_i)\setminus B_\rho (\hat x_i))} \to 0 \qquad \text{as }j \to \infty.
$$
In particular, since $u_\infty \in C^2(\Omega)$, there exists a constant $\bar C>0$ such that that following holds:
for any $\rho \in (0,r_0)$ there exists $j_\rho \in \mathbb N$ such that 
\begin{equation}
\label{eq:bound nabla}
\max_{1\leq i \leq \ell} \|\nabla u_j\|_{L^\infty(B_{r_0}(\hat x_i)\setminus B_\rho (\hat x_i))} \leq \bar C  \qquad \forall\,j\geq j_\rho.
\end{equation}
We now make the following:\\
{\bf Claim:} {\it There exist $\hat C,\hat r>0$ such that $\max_{1\leq i \leq \ell}\|u_j\|_{L^\infty(B_{\hat r}(\hat x_i))} \leq \hat C$ for all $j$ sufficiently large.}

Assuming for a moment that the claim is proved, since $u_j \to u_\infty$ in $C^2_{\rm loc}(\Omega\setminus \Sigma_\infty)$ and $u_\infty \in C^2(\Omega)$, it follows from the claim that
$$
\sup_{j}\|u_j\|_{L^\infty(\Omega_0)}<\infty,
$$
where $\Omega_0$ is given by Lemma~\ref{stable outside K}. But then Lemma~\ref{stable outside K} implies that 
$\sup_{j}\|u_j\|_{L^\infty(\Omega)}<\infty$, a contradiction to our initial assumption. Hence, in the case $f_\infty(0)>0$,
the theorem is proved provided we can show the claim.

To prove the claim, it suffices to control $\|u_j\|_{L^\infty(B_{r_0}(\hat x_i))}$ for each $i$. With no loss of generality, we can fix $i=1$ and assume that $\hat x_1=0$.
Then, thanks to \eqref{supercritical j} we can apply Proposition \ref{uniform integrable} to get the following estimate:
for any $r \in (0,r_0/2)$ and any $z \in B_r$, given $\rho \in (0,2r)$ it follows from \eqref{eq:bound nabla} that
\begin{multline}
\label{eq:decay rho}
\int_{B_r(z)}|\nabla u_j|^2\,dx \leq \int_{B_{2r}}|\nabla u_j|^2\,dx=\int_{B_{\rho}}|\nabla u_j|^2\,dx+\int_{B_r\setminus B_{\rho}}|\nabla u_j|^2\,dx\\
 \leq C\rho^{\delta}+\bar C |B_r|\leq C'\left(\rho^\delta +r^n\right)\qquad \forall\,j \geq j_\rho.
\end{multline}
Also, it follows from \eqref{eq:bound L1 j} that
\begin{equation}
\label{eq:bound L1 j bis}
\bint_{B_{r_0}}u_j\leq |B_{r_0}|^{-1}C_0 =:M, \qquad M=M(n,r_0,c_1,\Omega).
\end{equation}
Fix $\gamma:=1/2$, and let $m_0 \in \mathbb N$ be the constant provided by Proposition~\ref{prop:eps reg}
with $g=\hat{\lambda} \hat{f}$ (see \eqref{eq:def g}).
Then, with $C'$ as in \eqref{eq:decay rho}, we choose first $\bar r \in (0,r_0)$ such that $C'\bar r\le \frac12$, and then we fix $\rho \in (0,2\bar r)$ such that
$C'\rho^\delta \leq 2^{-m_0(n-1)-1}\bar r^{n-1}$.
With these choices it follows from \eqref{eq:decay rho} that, for any $r \in [2^{-m_0}\bar r,\bar r]$ and any $z \in B_{2^{-m_0}\bar r}$,
\begin{equation}
\label{eq:bound decay W12}
\int_{B_r(z)}|\nabla u_j|^2\,dx \leq C'\left(\rho^\delta +r^n\right) \leq 2^{-m_0(n-1)-1}\bar r^{n-1} +C'\bar r r^{n-1}\leq \frac12 r^{n-1}+\frac12 r^{n-1} =r^{n-1},
\end{equation}
provided $j$ is sufficiently large.
Hence, applying Proposition~\ref{prop:eps reg} to the functions $u_{j,z}(x):=u_j(z+\bar r x)$ with $f=\bar r ^2\lambda_j  f_j$ (note that $0 \leq \bar r ^2\lambda_j  f_j\leq \lambda_j  f_j\leq \hat{\lambda}  \hat{f}$),
thanks to \eqref{eq:bound L1 j bis} we conclude that $|u_j(z)|=|u_{j,z}(0)|\leq M+C(n)$ for all $j$ sufficiently large, for all $z \in B_{2^{-m_0}\bar r}$. Choosing $\hat r:=2^{-m_0}\bar r$, this proves the claim and concludes the proof of this case.

\subsection{The case $\lambda_\infty=0$}

Let $M_j=\|\nabla u_j\|_{L^2(\Omega)}.$
We prove the result by contradiction, distinguishing between two cases.

\smallskip

\noindent {\bf Case 1: ${M_j}\to 0$ as $j\to \infty$.} In this case, Proposition~\ref{prop:stability} implies that   $u_j\to 0$ in $C^2_{\rm loc}$  outside a set $\Sigma_\infty$ consisting of at most $k$ points.  
Since $\|\nabla u_j\|_{L^2(\Omega)} \to 0$, with the same notation as in the case $\lambda_\infty>0$, we deduce that \eqref{eq:bound decay W12} holds around each point $\hat x_i \in \Sigma_\infty\cap \Omega_0$. Also, by Poincar\'e and H\"older inequalities,
$$\|u_j\|_{L^1(\Omega)}\leq C(n,\Omega)\|u_j\|_{L^2(\Omega)}\leq C(n,\Omega)\|\nabla u_j\|_{L^2(\Omega)} \to 0.$$
Hence, arguing exactly as the previous case, thanks to Proposition~\ref{prop:eps reg} we deduce that  $|u_j(z)|\leq o(1)+C(n)$ for all $j$ sufficiently large, for all $z \in B_{2^{-m_0}\bar r}(\hat x_i)$. This implies that $\sup_j \|u_j\|_{L^\infty(\Omega)} <\infty$,
a contradiction.

\smallskip

\noindent {\bf Case 2: ${M_j}$ are uniformly bounded away from $0$.} 
Consider
$v_j:=\frac{u_j}{M_j}$, so that
\begin{equation}\label{vs}
-\Delta v_j=\lambda_j h_j(v_j),      \quad \quad       h_j(t):=M_j^{-1}f_j(M_j t),\qquad \|\nabla v_j\|_{L^2(\Omega)}=1.
\end{equation}
As in the proof of Proposition~\ref{prop:eps reg}, we multiply the equation satisfied by $v_j$
both by $v_j$ and by $x\cdot \nabla v_j$. Since $v_j \geq 0$, $v_j|_{\partial\Omega}=0$, and $\Omega$ convex,
as in the classical Derrick-Pohozaev argument (see, e.g., \cite[Proof of Theorem 1, Page 515]{E2010}) the boundary terms ``have the right sign'', and we get
$$ \frac{\lambda_j}{M_j^2}\int_{\Omega} f_j(M_j v_j)\,M_j v_j\, dx =\int_{\Omega} |\nabla v_j|^2 dx+o(1)=1+o(1),$$
and
$$\frac{\lambda_j}{M_j^2}\int_{\Omega}F_j(M_j v_j) \, dx \geq \frac {n-2}{2n} \int_{\Omega} |\nabla v_j|^2 dx+o(1)=\frac {n-2}{2n}+o(1). $$
Now, set $S_0:=\{x \in \Omega\,:\,M_j v_j \leq t_0\}$ and note that, since $f_j$ satisfies \eqref{F and f}, 
$$f_j(M_j v_j)\,M_j v_j \geq \left (\frac{2n}{n-2}+\ez\right)  F_j(M_j v_j) \quad \text{ in }  \Omega \setminus S_0.$$
Also, since $\lambda_j \to 0$ and $M_j$ is bounded away from $0$, 
$$
\frac{\lambda_j}{M_j^2}\int_{S_0} f_j(M_j v_j)\,M_j v_j\, dx \leq \frac{\lambda_j}{M_j^2}\int_{S_0} f_j(t_0)\,t_0\, dx\to 0,
$$
$$
\frac{\lambda_j}{M_j^2}\int_{S_0} F_j(M_j v_j)\, dx \leq \frac{\lambda_j}{M_j^2}\int_{S_0} F_j(t_0)\, dx\to 0.
$$
Therefore, combining all together,
\begin{multline*}
1+o(1)= \frac{\lambda_j}{M_j^2}\int_{\Omega\setminus S_0} f_j(M_j v_j)\,M_j v_j\, dx\\
\geq \left (\frac{2n}{n-2}+\ez\right) \frac{\lambda_j}{M_j^2}\int_{\Omega\setminus S_0 }F_j(M_j v_j) \, dx
\geq \left (\frac{2n}{n-2}+\ez\right) \frac {n-2}{2n}+o(1),
\end{multline*}
a contradiction for $j$ large enough, which concludes the proof of the theorem.

\appendix

\section{Boundedness of stable solutions in $B_2\setminus\{0\}$ for $3 \leq n \le 9$}
\label{removable}

It was shown in \cite{CC2006} that, if $3\leq n \leq9$ and $u\in W^{1,\,2}(B_2)\cap C^2(B_2\setminus\{0\})$ is a radially symmetric stable solution to \eqref{equ} in $\Omega=B_2\setminus \{0\}$, then
$$\|u\|_{L^{\infty}(B_{1})}\le C \|u\|_{L^1(B_2)}.$$
Namely, removing a point does not influence the interior estimate of the radially symmetric stable solutions. 

The aim of this appendix is to show that, combining the approximation argument in \cite{CC2006} with some modifications of the arguments in \cite{CFRS2020}, we can prove the following generalization of  \cite[Theorem 1.2]{CFRS2020} which is used in the proof of Proposition~\ref{prop:stability}:
\begin{prop}\label{removable singularity}
Let $3 \leq n \leq 9$, and let $u\in W^{1,\,2}(B_2)\cap C^2(B_2\setminus\{0\})$ be a stable solution to
$$
-\Delta u=f(u)\qquad \text{in }B_2\setminus \{0\},
$$
with
$f:\R\to \R$ nonnegative, increasing, and of class $C^1$.
Then 
$$\|u\|_{L^\infty (B_{1})}\le C \|u\|_{L^1(B_2)}$$
for some $0<\az<1$ and some universal constant $C>0.$
\end{prop}

We begin by proving the following generalization of \cite[Lemma 2.1]{CFRS2020} at the origin:

 \begin{lem}
\label{lem:stability estimates}
Let $3 \leq n \leq 9$, and let $u$ and $f$ be as in Proposition~\ref{removable singularity}.
Then
we have 
\begin{equation}\label{first order removable2}
   \int_{B_ {2 \rho/3} } |x\cdot \nabla u|^2 |x|^{-n} \, dx \le   C \rho^{2-n}\int_{B_\rho \setminus B_{2\rho/3}(y)} |\nabla u|^2  \,dx,
\end{equation}
and
\begin{equation}\label{second order removable}
\int_{B_ 1} |\nabla u|^2\,dx \le C  \int_{B_{3/2}\setminus B_1} |\nabla u|^2 \,dx.
\end{equation}
\end{lem}

\begin{proof}
For simplicity of notation, given $0<r<s<1$, we define $A(s,r):=B_r\setminus B_s.$


Fix $0<\theta \ll \ez \ll \rho$, $\eta\in C_c^1(A(\theta,2))$, and consider $\xi=(x\cdot \nabla u)\eta$ as test function in the stability inequality for $u$. Then, by the very same computation as the one in \cite[Proof of Lemma 2.1, Step 1]{CFRS2020},
we have 
\begin{equation}\label{first order removable}
0\le \int_{A(\theta,\,2)}\Big((x\cdot \nabla u)^2 |\nabla\eta|^2 +2(x\cdot \nabla u) \nabla u\cdot \nabla(\eta^2)-|\nabla u|^2 x\cdot \nabla (\eta^2)-(n-2)|\nabla u|^2\eta^2\Big)\, dx.
\end{equation}

We first choose
$\eta\in C^1_c(A(\theta,2))$ so that $0 \leq \zeta \leq 1$, $\zeta=1$ inside $A(2\theta, 1)$, $|\nabla \zeta|\leq C/\theta$ in $A(\theta, 2\theta)$, and $|\nabla \zeta|\leq C $ in $A(1, 3/2)$.
Then it follows from Young's inequality that
$$\int_{A(\theta,\,2)} |Du|^2 \le C  \int_{A(\theta,\,2\theta)} |x|^2|Du|^2\theta^{-2} + C \int_{A(1, 3/2)} |Du|^2.$$
Then \eqref{second order removable} follows by sending $\theta\to 0$ together with the assumption that $u\in W^{1,\,2}(B_2).$

As for \eqref{first order removable2}, 
we choose
$\eta=\min\big\{|x|^{1-\frac n 2},  \ez^{1-\frac n 2}\big\} \zeta$ with
$\zeta\in C^1_c(A(\theta,2))$, then 
inside $A(\ez,2)$ the formulas are identical to the ones in \cite[Proof of Lemma 2.1, Step 2]{CFRS2020}. Therefore, in the integrals over $A(\ez,2)$ we have exactly all the terms appearing in \cite[Equation (2.2)]{CFRS2020}, and the only difference concerns the integrals over $A(\theta,\ez)$.
Note that, inside $A(\theta,\ez)$, it holds
$$
|\nabla\eta|^2= \ez^{2-n}|\nabla \zeta|^2,\qquad \nabla(\eta^2)=2\ez^{2-n}\zeta\nabla \zeta.
$$
Thus, we obtain
\begin{align*}
0&\leq \int_{A(\ez,\,2)}\Big(-\frac{(n-2)(10-n)}{4}  |x|^{-n}(x\cdot \nabla u)^2 \zeta^2 -2|\nabla u|^2 |x|^{2-n} \zeta \ x\cdot \nabla\zeta + 4 |x|^{2-n} (x\cdot \nabla u)\zeta \ \nabla u\cdot   \nabla\zeta \nonumber\\
&\quad \quad \quad \quad \quad \quad \quad \quad \quad \quad -(n-2)|x|^{-n} \zeta (x\cdot \nabla\zeta)(x\cdot \nabla u)^2 +|x|^{2-n}(x\cdot \nabla u)^2  |\nabla\zeta|^2\Big)\,dx \\
&+\ez^{2-n}\int_{A(\theta,\,\ez)}\Big((x\cdot \nabla u)^2|\nabla \zeta|^2+4(x\cdot \nabla u)\zeta\,\nabla u\cdot \nabla \zeta
-2|\nabla u|^2\zeta (x\cdot \nabla \zeta)-(n-2)|\nabla u|^2\zeta^2\Big)\,dx.
\end{align*}
We now choose $\zeta \in C^1_c(A(\theta,\rho))$ such that $0 \leq \zeta \leq 1$, $\zeta=1$ inside $A(2\theta,\rho/2)$, $|\nabla \zeta|\leq C/\theta$ in $A(\theta,2\theta)$, and $|\nabla \zeta|\leq C/\rho$ in $A(7\rho/8,\rho)$.
With this choice, the formula above implies
\begin{align*}
\frac{(n-2)(10-n)}{4} \int_{A(\ez,\,7\rho/8)} |x|^{-n}(x\cdot \nabla u)^2 \, dx
& \le  C\rho^{2-n} \int_{A(7\rho/8,\,\rho)}  |\nabla u|^2 + C\ez^{2-n}\int_{A(\theta,\,2\theta)}|\nabla u|^2\,dx,
\end{align*}
so \eqref{first order removable2} follows by letting first $\theta \to 0$ and then $\ez \to 0$ (recall that $3\leq n \leq 9$).

 
\end{proof}

We can now prove the following analogue of \cite[Lemma 3.1]{CFRS2020}:\footnote{In Lemma~\ref{lem:22} we require $3 \leq n \leq 9$ since we proved \eqref{second order removable} as a consequence of \eqref{first order removable}, and the latter bound requires this dimensional restriction.
However, for $n\geq 10$ one could combine our approximation argument with \cite[Proof of Theorem 7.1]{CFRS2020} to show that
$$
   \int_{B_ {7\rho/8}(y)} |(x-y)\cdot \nabla u|^2 |x-y|^{-a} \, dx \le   C \rho^{2-a}\int_{B_\rho(y)\setminus B_{7\rho/8}(y)} |\nabla u|^2  \,dx.
$$
for any $a<2(1+\sqrt{n-1}).$ In particular, choosing $a=2$ and arguing as in the proof of Lemma~\ref{lem:stability estimates}, one proves the validity of \eqref{second order removable} in every dimension. As a consequence, one  can show that Lemma~\ref{lem:22} holds in every dimension.}
\begin{lem}\label{lem:22}
Let $3 \leq n \leq 9$, and let $u\in W^{1,\,2}(B_2)\cap C^2(B_2\setminus \{0\})$ be a stable solution to
$$
-\Delta u=f(u)\qquad \text{in }B_2\setminus \{0\},
$$
with
$f:\R\to \R$ nonnegative, increasing, and of class $C^1$.
Assume that
$$
\int_{B_1}|\nabla u|^2\,dx\geq \delta \int_{B_2}|\nabla u|^2\,dx
$$
for some $\delta>0$.
Then there exists a constant $C_\delta$ such that
$$
\int_{B_{3/2}\setminus B_1}|\nabla u|^2 \,dx\leq C_\delta \int_{B_{3/2}\setminus B_1}|x\cdot \nabla u|^2\,dx.
$$
\end{lem}

\begin{proof}
Assume the result to be  false.
Then, there exists a sequence of stable solutions to $-\Delta u_k=f_k(u_k)$ in $B_2\setminus \{0\}$, with $f_k:\R\to \R$ nonnegative, increasing, and of class $C^1$,
such that
\begin{equation}\label{eq:doubling}
\int_{B_1}|\nabla u_k|^2\,dx\geq \delta \int_{B_2}|\nabla u_k|^2\,dx,
\qquad 
\int_{B_{3/2}\setminus B_1} |\nabla u_k|^2\,dx =1,
\quad\text{and}\quad \int_{B_{3/2}\setminus B_1}|x\cdot \nabla u_k|^2\,dx\to 0.
\end{equation}
Now, thanks to \eqref{eq:doubling} and \eqref{second order removable},
\begin{equation}\label{3.1bis}
 \int_{B_2}|\nabla u_k|^2\,dx\leq \frac{1}{\delta} \int_{B_1}|\nabla u_k|^2\,dx\leq   \frac{C}{\delta}\int_{B_{3/2}\setminus B_1} |\nabla u_k|^2\,dx = \frac{C}{\delta}.
\end{equation}
Therefore, since $u_k$ is stable in $B_2\setminus B_{1/2}$, it follows from \cite[Proposition 2.4]{CFRS2020} and a standard scaling and  covering argument that 
\[
 \|\nabla u_k\|_{L^{2+\gamma}({B_{3/2}\setminus B_{1}})} \le C \|\nabla u_k\|_{L^{2}({B_{2}\setminus B_{1/2}})} \leq \frac{C}{\delta}.
\]
This implies that the sequence of superharmonic functions 
\[
v_k : = u_k -\ave_{B_{3/2}\setminus B_1}u_k
\]
satisfies  $$\|v_k\|_{L^1(B_{3/2}\setminus B_1)} \leq C\|v_k\|_{L^2(B_{3/2}\setminus B_1)} \leq C$$
(thanks to \eqref{3.1bis} and H\"older together with Poincar\'e inequalities), as well as
\[
\|\nabla v_k\|_{L^2(B_{3/2}\setminus B_1)}=  1,\quad  \| v_k\|_{W^{1,2+\gamma}(B_{3/2}\setminus B_1)}\le C,
\qquad
\int_{B_{3/2}\setminus B_1}|x\cdot \nabla v_k|^2\,dx\to 0.
\]
Thus, as in the proof of \cite[Proposition 2.4]{CFRS2020}, up to a subsequence we have that $v_k \to  v$ strongly in $W^{1,2}(B_{3/2}\setminus B_1)$, where $v$ is a superharmonic function in $B_{3/2}\setminus B_1$ satisfying 
\[
\|\nabla v\|_{L^2(B_{3/2}\setminus B_1)}=  1 \qquad \mbox{and} \qquad x\cdot \nabla v \equiv 0 \quad \mbox{a.e. in }B_{3/2}\setminus B_1.
\]
Again as in the proof of \cite[Proposition 2.4]{CFRS2020}, this implies that $v$ is constant in $B_{3/2}\setminus B_1$, a contradiction that proves the result.
\end{proof}

We can now prove the main result of this appendix.

\begin{proof}[Proof of Proposition~\ref{removable singularity}]
We divide the proof into two steps.

\noindent{\bf Step 1: Gradient decay at the origin}. 
The argument is similar to the one in \cite[Proof of Theorem 1.2]{CFRS2020}, with some minor modifications.
Given  $\rho \in (0,1)$ we define the 
quantities 
$$
\mathcal D(\rho):=\rho^{2-n}\int_{B_\rho}|\nabla u|^2\,dx
\qquad \textrm{and} \qquad
\mathcal R(\rho):=\int_{B_\rho }|x|^{-n}|x\cdot \nabla u|^2\,dx .
$$
We first claim that there exists a dimensional exponent $\alpha>0$ such that
\begin{equation}
\label{eq:decay R}
\mathcal R(\rho)\leq C \|\nabla u\|^2_{L^2(B_{3/2})}\rho^{2\alpha}\qquad \forall\,\rho \in (0,1/4).
\end{equation}

To prove this claim, note that \eqref{first order removable2} implies that
\begin{equation}\label{agnwiohwio}
\mathcal R(\rho)\leq C\rho^{2-n}\int_{B_{3\rho/2}\setminus B_\rho}|\nabla u|^2\,dx \qquad \forall\,\rho \in (0,1/4).
\end{equation}
Hence, if $\mathcal D(\rho)\geq \frac12 \mathcal D(2\rho)$,
then we can apply Lemma \ref{lem:22} with $\delta=1/2$ to the function $u$ at the origin to deduce that
$$
\rho^{2-n}\int_{B_{3\rho/2}\setminus B_\rho}|\nabla u|^2\,dx \leq C\rho^{-n} \int_{B_{3\rho/2}\setminus B_\rho}|x\cdot \nabla u|^2\,dx\leq C \bigl(\mathcal R(3\rho/2)-\mathcal R(\rho)\bigr)
$$
for some universal constant $C$; recall $3\le n\le 9$.
Combining this bound with \eqref{agnwiohwio} and using that $\mathcal R$ is nondecreasing, we deduce that
\begin{equation}\label{ahsiogwio}
\mathcal R(\rho,y)\leq C\bigl(\mathcal R(2\rho)-\mathcal R(\rho\bigr)
\qquad \text{provided $ \mathcal D(\rho)\geq {\textstyle \frac12}  \mathcal D(2\rho)$. }
\end{equation}

Thus if we define $a_j : = \mathcal D(2^{-j-2})$, $b_j := \mathcal R(2^{-j-2})$, then there exists a universal constant $L>1$ such that: 
\begin{itemize}
\item[(i)] $b_j\le b_{j-1}$ for all $j\ge 1$ (since $\mathcal R$ is nondecreasing); 
\item[(ii)] $a_j+ b_j \le L a_{j-1}$   for all $j\ge 1$ (by \eqref{agnwiohwio});
\item[(iii)] if $a_{j}  \ge \frac 1 2  a_{j-1}$ then $b_{j} \le  L(b_{j-1} -b_j)$  for all  $ j\in \mathbb N$. 
(by \eqref{ahsiogwio}).
\end{itemize}
Therefore, if we choose $\ez>0$ such that $2^{-\ez}=\frac{L^{1+\ez}}{1+L}$, and define $c_j:=a_j^\ez b_j$ and $\theta:=(2^{-\ez})^{\frac{1}{1+\ez}} \in (0,1)$, then the proof of \cite[Lemma 3.2]{CFRS2020} 
shows that 
$$
c_{j+1}\leq \theta c_j \qquad \text{for all  $ j\in \mathbb N$},
$$
which implies that
\begin{equation}
\label{eq:geom 1}
c_j \leq \theta^{j}c_0 \quad \text{for }  j\ge 1. 
\end{equation}
As in the proof of \cite[Lemma 3.2]{CFRS2020}, this implies that 
$$
b_j \leq C(a_0 + b_0)\theta^j\leq C \|\nabla u\|^2_{L^2(B_{3/2})}\theta^j\qquad \forall\,j \geq 1,
$$
so \eqref{eq:decay R} follows by choosing $\alpha>0$ so that $2^{-2\alpha}=\theta$.

We now observe that, thanks to \cite[Proposition 2.5]{CFRS2020} together with a standard scaling and covering argument, we have
$
\|\nabla u\|_{L^2(B_{3/2}\setminus B_1)}\leq C\|u\|_{L^1(B_{2}\setminus B_{1/2})}.
$
Hence, combining this bound with \eqref{eq:decay R} and \eqref{second order removable},
we obtain
\begin{equation}
\label{eq:decay R2}
\mathcal R(\rho)\leq C\|u\|^2_{L^1(B_{2})}\rho^{2\alpha}\qquad \forall\,\rho \in (0,1/4).
\end{equation}
This concludes the first step. 

\noindent{\bf Step 2: Boundedness inside the unit ball}.
First, using polar coordinates $(s,\omega)$ and Fubini's Theorem,
\begin{align*}
\int_{B_\rho}|u(x)-u(0)|\,dx
&=\int_{\mathbb S^{n-1}}\int_0^\rho |u(s\omega)-u(0)|\,s^{\,n-1}\,ds\,d\omega \\
&\le \int_{\mathbb S^{n-1}}\int_0^\rho\!\!\left(\int_0^s |\nabla u(t\omega)\cdot\omega|\,dt\right)s^{\,n-1}\,ds\,d\omega \\
&= \int_{\mathbb S^{n-1}}\int_0^\rho |\nabla u(t\omega)\cdot\omega|
      \left(\int_t^\rho s^{\,n-1}\,ds\right)dt\,d\omega
 \le \frac{\rho^n}{n}\int_{B_\rho}\frac{|\nabla u(x)\cdot x|}{|x|^n}\,dx .
\end{align*}
Next, using H\"older inequality,
\begin{align*}
\int_{B_\rho}\frac{|x\cdot\nabla u|}{|x|^n}\,dx
&=\sum_{k\ge0}\int_{B_{2^{-k}\rho}\setminus B_{2^{-(k+1)}\rho}}\frac{|x\cdot\nabla u|}{|x|^n}\,dx \\
&\le \sum_{k\ge0}\bigg(\int_{B_{2^{-k}\rho}\setminus B_{2^{-(k+1)}\rho}}\frac{|x\cdot\nabla u|^2}{|x|^n}\,dx\bigg)^{1/2}
          \bigg(\int_{B_{2^{-k}\rho}\setminus B_{2^{-(k+1)}\rho}}|x|^{-n}\,dx\bigg)^{1/2} \\
          &=C  \sum_{k\ge0}\bigg(\int_{B_{2^{-k}\rho}\setminus B_{2^{-(k+1)}\rho}}\frac{|x\cdot\nabla u|^2}{|x|^n}\,dx\bigg)^{1/2}
          \le  C \sum_{k\ge0}\sqrt{\mathcal R(2^{-k}\rho)} .
\end{align*}
Thus, combining the last two equations with \eqref{eq:decay R2}, we obtain 
\begin{equation}\label{holder average at 0}
     \bint_{B_\rho} |u(x)-u(0)|\,dx \leq C  \|u\|_{L^1(B_{2})} \sum_{k=0}^\infty (2^{-k}\rho)^\az \leq C \|u\|_{L^1(B_{2})} \rho^\alpha  \qquad \forall\,\rho \in (0,1/4)
\end{equation}
Next, given $y\in B_{1/8}$, define $r=|y|$. Then the interior estimate from \cite[Theorem 1.2]{CFRS2020} applied to the rescaled function $u(y+r\,\cdot )-u(0)$ yields
\begin{multline}
\label{eq:uy0}
|u(y)-u(0)|\leq \|u(y+r\,\cdot )-u(0)\|_{L^\infty(B_{1/2})}
\\
\leq C \|u(y+r\,\cdot )-u(0)\|_{L^1(B_{1})}
=C \bint_{B_r(y)} |u(x)-u(0)|\,dx.
\end{multline}
Noticing that $B_r(y)\subset B_{2r}$, \eqref{holder average at 0} applied with $\rho=2r$ implies
$$
\bint_{B_r(y)} |u(x)-u(0)|\,dx \leq C \bint_{B_{2r}(0)} |u(x)-u(0)|\,dx \leq C \|u\|_{L^1(B_{2})} r^\alpha\leq C \|u\|_{L^1(B_{2})},
$$
that combined with \eqref{eq:uy0} gives
$$
|u(y)-u(0)| \leq  C \|u\|_{L^1(B_{2})} \qquad  \forall\, y \in B_{1/8}.
$$
Furthermore, using \eqref{holder average at 0} with $\rho=1/4$,
$$
|u(0)|\leq \bint_{B_{1/4}} |u(x)-u(0)|\,dx+\bint_{B_{1/4}} |u(x)|\,dx\leq C\|u\|_{L^1(B_{2})}.
$$
Combining the two equations above, we finally get
$$
\|u\|_{L^\infty(B_{1/8})}\le C \|u\|_{L^1(B_{2})} .
$$
Since the $L^\infty$ norm of $u$ inside $B_1\setminus B_{1/8}$ can be directly bounded using \cite[Theorem 1.2]{CFRS2020}, the result follows.
\end{proof}

\section{Uniform boundedness of solutions with spectrum bounded below}
\label{finite morse L}
Although not relevant for this paper, it is interesting to observe that, by simply adapting the arguments in \cite{CFRS2020}, one can deduce an a priori bound in $L^\infty$ for solutions of $-\Delta u=f(u)$ whenever the spectrum of the linearized operator $-\Delta -f'(u)$ is contained inside $[-\Lambda,+\infty)$ for some finite constant $\Lambda \geq 0$. Also, for finite Morse index solutions, the constant $\Lambda$ depends only on $n$ and on a maximal finite dimensional subspace $X_k$ on which $Q_u$ is negative definite.
Unfortunately one cannot hope in general to control $\Lambda$ in terms only on the Morse index, as can be seen by considering the family of solutions \eqref{critical case} (which has index $1$).

To present this result, consider 
 $u \in C^2(B_2)$ a solution to $-\Delta u=f(u)$ in $B_2$ with ${\rm ind}(u,B_2)\leq k$,
and define
$$
\widehat Q_u[\xi,\zeta] :=\int_{B_{2}}\Big(\nabla\xi\cdot \nabla \zeta -f'(u)\,\xi\,\zeta\Big)\, dx.
$$
Since ${\rm ind}(u,B_{1})\leq k$, there exists a $k$-dimensional set $X_k \subset C^1_c(B_{1})$ such that, for any $\xi \in C^1_c(B_{1})$, we can write $\xi=\xi_k+\xi'$ with $\xi_k \in X_k$, $\widehat Q_u[\xi',\xi']\geq 0$, and $\widehat Q_u[\xi',\xi_k]= 0$.

Now, since $X_k$ is finite dimensional, 
$$ 
\sup_{\xi \in X_k,\,\|\xi\|_{L^2(B_{1})}}\|\xi\|_{L^\infty(B_{1})}=:A_k<\infty,
$$
so it follows from Lemma~\ref{lem:Morse}(ii) (and a covering argument) that
$$
\inf_{\xi \in X_k,\,\|\xi\|_{L^2(B_{1})}=1}\int_{B_{1}} \Bigl(|\nabla\xi|^2-f'(u)\xi^2\Bigl)\, dx 
\geq - \sup_{\xi \in X_k,\,\|\xi\|_{L^2(B_{1})}}\|\xi\|_{L^\infty(B_{3/4})}^2\int_{B_{1}}f'(u)\,dx \geq  -C_n A_k^2,
$$
which implies that
\begin{equation}\label{finite morse index}
\int_{B_1} |\nabla\xi|^2\, dx \geq \int_{B_1} \big(f'(u)-\Lambda\big)\xi^2\, dx \qquad \forall\,\xi \in C^1_c(B_{1}),
\end{equation}
where $\Lambda:=C_n A_k^2$.
In other words, the spectrum of the operator $-\Delta -f'(u)$ on $L^2(B_1)$ is bounded from below by $-\Lambda$.

In \cite[Theorem 1.1]{CFRS2020}, whenever $3\leq n \leq 9$, the authors proved an a priori $L^\infty$ estimate\footnote{Actually, \cite[Theorem 1.1]{CFRS2020} provides a universal $C^\alpha$ bound for some $\alpha>0$. Analogously, also in the general case $\Lambda \geq 0$ one can prove an interior bound on $\|u\|_{C^\alpha}$.} for solutions of $-\Delta u=f(u)$ satisfying \eqref{finite morse index} with $\Lambda=0$.
The goal of this appendix is to show how to extend such a result to the general case $\Lambda\geq  0.$

\begin{prop}\label{finite depend L}
Let $3 \leq n \leq 9$, let
$f:\R\to \R$ be nonnegative, increasing, and of class $C^1$, and let $u \in C^2(B_2)$ solve $-\Delta u=f(u)$ 
and satisfy
\eqref{finite morse index} for some $\Lambda \geq 0$.
Then
$$\|u\|_{L^\infty (B_{1/2})}\le C(\Lambda) \|u\|_{L^1(B_1)}.$$
\end{prop}

The proof of Proposition~\ref{finite depend L} is very similar to that in \cite[Sections 2 \& 3]{CFRS2020}, the main differences being in two interior estimates that we present here. Once the two lemmas below are available, the proof follows by the same argument as in \cite{CFRS2020}, and we leave the details to the interested reader.

\begin{lem}\label{radial gradient}
Let $u\in C^2(B_1)$ be as in Proposition~\ref{finite depend L}. Then, for any $\eta\in  C^1_c(B_1)$, we have
\begin{equation}\label{radial test}
\int_{B_1} ((n-2)\eta+2x\cdot \nabla\eta)\eta|\nabla u|^2-2(x\cdot \nabla\eta)\nabla u \cdot \nabla(\eta^2)-|x\cdot \nabla u|^2(|\nabla\eta|^2+\Lambda\eta^2) \, dx \le 0.
\end{equation}
Thus, for any $\varphi\in C^1_c(B_1)$, we have
\begin{align}\label{radial test varphi}
&\frac {(n-2)(10-n)}{4}\int_{B_1}|x|^{-n}|x\cdot \nabla u|^2(1-\Lambda|x|^2)\varphi^2 \, dx \nonumber\\
 &\le \int_{B_1}\Big(-2 |x|^{2-n}|\nabla u|^2\varphi (x\cdot \nabla\varphi)+4|x|^{2-n}(x\cdot \nabla u)\varphi \nabla u\cdot \nabla\varphi\, dx\nonumber\\ 
 &\qquad \qquad  +(2-n)|x|^{-n}|x\cdot \nabla u|^2 \varphi (x\cdot \nabla\varphi)   + |x|^{2-n}|x\cdot \nabla u|^2 |\nabla\varphi|^2\Big) \, dx
\end{align}
In particular, for $0<r<\frac12 \min\left\{1,\Lambda^{-1/2}\right\}$,
\begin{equation}\label{radial caccio}
\int_{B_r} |x|^{-n}|x\cdot \nabla u|^2\,dx \le C(n) r^{2-n} \int_{B_{3r/2}\setminus B_{r}} |\nabla u|^2\, dx.
\end{equation}
\end{lem}
\begin{proof}
We proceed as in \cite[Lemma 2.1]{CFRS2020} and sketch the proof here. First we choose $\xi=(x\cdot \nabla u)\eta$ in \eqref{finite morse index}, with $\eta \in  C^1_c(B_1)$, to get
\begin{equation}
\int_{B_1}\left( \Delta (x\cdot \nabla u) +f'(u)(x\cdot \nabla u)\right) (x\cdot \nabla u)\eta^2\,dx \le \int_{B_1} (x\cdot \nabla u)^2 (|\nabla\eta|^2+\Lambda\eta^2)\, dx.
\end{equation}
Then by noticing that
\begin{equation}
\label{eq:xDu}
\Delta (x\cdot \nabla u)= x\cdot \nabla \Delta u+2 \Delta u=-f'(u) (x\cdot \nabla u)+2\Delta u,
\end{equation}
we conclude \eqref{radial test}. Then \eqref{radial test varphi} follows by taking $\eta=|x|^{-\frac {n-2} 2} \varphi$, and  \eqref{radial caccio} follows by further choosing $\varphi$ as a cut-off function supported in  $B_{3r/2}$ with $\varphi=1$ on $B_r$.
\end{proof}

\begin{lem}\label{second order}
Let $u\in C^2(B_1)$ be as in Proposition~\ref{finite depend L}. Write
$$\mathcal A=\left( |D^2 u|^2- \frac{|D^2 u \cdot\nabla u|^2}{|\nabla u|^2} \right)^{\frac 1 2} \, \text{ when $|\nabla u|\neq 0$,}\qquad\text{ and } \ \mathcal A=0 \ \text{ when $|\nabla u|= 0$. } $$
Then, for any $\eta \in C^{1}_c(B_1)$, we have
$$\int_{B_1} \mathcal A^2 \eta^2\,dx \le (1+\Lambda)\int_{B_1} |\nabla u|^2 |\nabla\eta|^2\,dx$$
\end{lem}
\begin{proof}
We follow the argument of \cite[Lemma 2.3]{CFRS2020} and again sketch the proof.
Set $u_i:=\partial_iu$.
Multiplying  both side of the equation
$-\Delta u_i =f'(u) u_i$
by $u_i\eta^2$, and summing over  $i=1,\ldots,n$, we get
$$\int_{B_1} \bigg(\sum_i |\nabla(u_i\eta)|^2- |\nabla u|^2|\eta|^2\bigg)\, dx= \int_{B_1} f'(u) |\nabla u|^2 \eta^2\, dx. $$
On the other hand, choosing $\xi=|\nabla u|\eta$ in \eqref{finite morse index}, we have
$$\int_{B_1} |\nabla(|\nabla u|\eta)|^2+\Lambda|\nabla u|^2 \eta^2\,dx \ge \int_{B_1} f'(u)|\nabla u|^2 \eta^2\, dx.$$
Thus we obtain
$$\int_{B_1}|\nabla u|^2|\eta|^2 +\Lambda|\nabla u|^2 \eta^2\,dx \ge \int_{B_1} \left(\sum_i |\nabla(u_i\eta)|^2-|\nabla(|\nabla u|\eta)|^2\right)\, dx, $$
and we aconclude the lemma by noticing that
$$ \sum_i |\nabla(u_i\eta)|^2-|\nabla(|\nabla u|\eta)|^2=\mathcal A^2\eta^2.$$
\end{proof}

\end{document}